\documentclass[11pt,a4paper,reqno]{amsart}

\usepackage{amsmath,amssymb,amsthm,amsfonts}
\usepackage{amsbsy}
\usepackage[utf8]{inputenc}
\usepackage{geometry}
\usepackage{a4wide}
\usepackage[english]{babel}
\usepackage{lmodern}

\usepackage{babelbib}
\usepackage{cancel}

\usepackage{indentfirst}

\usepackage{geometry}
\usepackage{amsfonts}
\usepackage{graphicx}

\usepackage{float}
\usepackage{xcolor}
\usepackage{hyperref}
\usepackage{cleveref}
\hypersetup{colorlinks,linktocpage}
\usepackage{verbatim}
\usepackage{enumerate}
 \usepackage[normalem]{ulem}

\newcommand{\V}{\mathbb{V}}

\newcommand{\Z}{{\mathbb Z}}

\newcommand{\Q}{{\mathbb Q}}
\newcommand{\R}{{\mathbb R}}
\newcommand{\C}{{\mathbb C}}

\newcommand{\F}{{\mathcal F}}

\newcommand{\Sp}{{\mathbb S}}

\newcommand{\Hy}{{\mathbb H}}

\newcommand{\h}{{\mathcal{H}}}  
\newcommand{\Cliff}{{\mathcal{C}}}  
\newcommand{\SB}{{\rm SB}}   

\newcommand{\bc}{\begin{center}}
\newcommand{\ec}{\end{center}}

\newcommand{\U}{{\mathcal U}}

\newcommand{\GL}{{\rm GL}}
\newcommand{\SL}{{\rm SL}}

\newcommand{\PSL}{{\rm PSL}}
\newcommand{\SU}{{\rm SU}}
\newcommand{\PSU}{{\rm PSU}}

\newcommand{\Iso}{{\rm Iso}}

\renewcommand{\O}{\mathcal{O}}

\newcommand{\BQ}{{\mathbb B}}

\newcommand{\Isom}{{\rm Isom}}

\newcommand{\SQ}{{\mathbb S}}

\newtheorem{theorem}{Theorem}[section]
\newtheorem{definition}[theorem]{Definition}
\newtheorem{lemma}[theorem]{Lemma}
\newtheorem{corollary}[theorem]{Corollary}
\newtheorem{proposition}[theorem]{Proposition}
\newtheorem{remark}[theorem]{Remark}

\newtheorem{algorithm}[theorem]{Algorithm}
\newtheorem{method}[theorem]{Method}

\title[On units in orders in $2$-by-$2$ matrices over quaternion algebras]{On units in orders in $2$-by-$2$ matrices over quaternion algebras with rational center}

\date{\today}

\thanks{During the writing process, the  author was supported by the University of Bielefeld and the research foundation FWO (Fonds voor
Wetenschappelijk Onderzoek Flanders).}

\author[A.~Kiefer]{Ann Kiefer}
\address{Department Wiskunde
Vrije Universiteit Brussel
Pleinlaan, 2
1050 Brussel
Belgium}
\email{ann.kiefer@vub.be}

\begin{document}

\begin{abstract}
We generalize an algorithm established in earlier work \cite{algebrapaper} to compute finitely many generators for a subgroup of finite index of an arithmetic group acting properly discontinuously on hyperbolic space of dimension $2$ and $3$, to hyperbolic space of higher dimensions using Clifford algebras.
 We hence get an algorithm which gives a finite set of generators of finite index subgroups of a discrete subgroup of Vahlen's group, i.e. a group of $2$-by-$2$ matrices with entries in the Clifford algebra satisfying certain  conditions. 
 The motivation comes from units in integral group rings and this new algorithm allows to handle unit groups of orders in $2$-by-$2$ matrices over rational quaternion algebras. The rings investigated are part of the so-called exceptional components of a rational  group algebra.  
\end{abstract}

\maketitle

\section{Introduction}

The aim of this paper is to present an algorithm computing finite generating sets of finite index subgroups of the groups $\SL_2(\O)$, where $\O$ is an order in a $2$-by-$2$ matrix algebra over a totally definite quaternion algebra with rational centre, i.e. over $(\frac{x,y}{\Q})$, with $x$ and $y$ negative integers. 
The motivation comes from the problem of describing finitely many units in the integral group ring $\Z G$ of a finite group $G$ that generate a subgroup of finite index in the unit group $\U (\Z G)$ of $\Z G$ (for more details see \cite{ericangelbook} and \cite{ericangelbook2}). 

The general idea is the following. To study the unit group $\U(\Z G)$ of the integral group ring $\Z G$, we consider the rational group ring $\Q G$. By Wedderburn-Artin's Theorem, $\Q G$ may be written as a direct product of matrix rings over division algebras, i.e. 
\begin{equation}\label{Wedderburn}
\Q G \cong \Pi_{i=1}^r M_{n_i}(D_i),
\end{equation}
where $D_i$ is a division algebra for every $1 \leq i \leq r$. We construct an order in the right hand side of (\ref{Wedderburn}). Therefore,  for $1 \leq i \leq r$, let $\O_i$ be some order in $D_i$. Then clearly $\Pi_{i=1}^r M_{n_i}(\O_i)$ is an order in $\Pi_{i=1}^r M_{n_i}(D_i)$. The unit group of the latter order is given by $\Pi_{i=1}^r \GL_{n_i}(\O_i)$. The group $\U(\Z G)$ is the group of units of an order in the left hand side of (\ref{Wedderburn}), while $\Pi_{i=1}^r \GL_{n_i}(\O_i)$ is the group of units of an order in the right hand side. It is well known that the groups of units of two orders in the same finite dimensional rational algebra are commensurable, i.e. they have a common subgroup that is of finite index in both. For more details we refer an interested reader to \cite[Chapter 2.9]{sehmil}. Moreover $\U(\mathcal{Z}(\O_i)) \times \SL_{n_i}(\O_i)$, the direct product of the central units in $\O_{i}$ and  the group consisting of the reduced norm one elements in $M_{n_{i}}(\O_{i})$, contains a subgroup of finite index isomorphic to a subgroup of finite index of $\GL_{n_i}(\O_i)$.
The group $\prod_{i=1}^{n} \U(\mathcal{Z}(\O_i))$ may be determined from Dirichlet's Unit Theorem and  it is commensurable with $\mathcal{Z} (\U (\Z G))$, the group of central units of $\Z G$.
For a large class of finite groups $G$ one can describe generators of a subgroup of finite index in  $\mathcal{Z} (\U (\Z G))$ \cite{JOdRV} (see also \cite{JdRV}).

Hence, up to commensurability, the problem of finding generators and relators for $\U(\Z G)$ reduces to finding a presentation of $\SL_{n_i}(\O_i)$ for every $1 \leq i \leq n$.
The congruence theorems allow to compute generators of finite index subgroups of $\SL_{n_i}(\O_i)$ 
when   $n_i \geq 3$ (without any further restrictions) and also  for $n_i=2$  and $n_i=1$, except if $M_{n_i}$ is a so-called exceptional component. Exceptional components are 
\begin{enumerate}
\item $M_2(\Q)$,
\item $M_2(\Q(\sqrt{-d}))$ with $d$ a positive square-free integer,
\item $M_2((\frac{x,y}{\Q}))$ with $(\frac{x,y}{\Q})$ a totally definite quaternion algebra and
\item non commutative division algebras $D_i$ that are not totally definite quaternion algebras. 
\end{enumerate}
For details, we refer the interested reader to \cite[Chapter 12]{ericangelbook}.

Hence the problem reduces to finding generators of a finite index subgroup of $\SL_2(\O)$ or $\U_1(\O)$, i.e. units of reduced norm $1$ in $\O$, for $\O$ an order in the division algebras listed under (1) to (4). One idea to attack these groups of units is to consider their group action on an adequate space. In \cite{PR}, Pita, del R\'{\i}o and Ruiz did this for the unit group of orders in small exceptional components of type (1) and (2), by letting them act on hyperbolic spaces of dimension $2$ and $3$. In \cite{jesetall}, the authors did the same thing but in the context of the unit group of a maximal order in a quaternion algebra. This constituted the first concrete example of an attack on the exceptional components of type (4). This was then finally generalized in \cite{algebrapaper}, where an algorithm, called the DAFC, is given to compute an 'approximate' fundamental domain, and hence generators, for subgroups of finite index of groups that are contained in the unit group of orders in quaternion algebras over $\Q$ or quadratic imaginary extensions of $\Q$, in particular for Bianchi groups.  As $\PSL_2(\R)$ and $\PSL_2(\C)$ are groups of orientation preserving isometries of the upper half-space model of hyperbolic spaces of dimension $2$ and $3$ respectively, the DAFC algorithm is established in the upper half-space model of hyperbolic space. As the ball model of hyperbolic space is more symmetric, some computations are carried out in the ball model. For better visualisation, the main algorithm is done in the upper half-space model. This work handled the exceptional components of type (1) and (2), as well as a part of exceptional components of type (4). The algorithm is based on Poincar\'e's Polyhedron Theorem, but does not use it in its most general version, as only an 'approximate' fundamental domain is necessary.  Algorithms that compute exact fundamental domains may be found in \cite{riley} and \cite{Page15}. In \cite{BCNS15} and \cite{JKdR16} first attempts were done to treat the unit group of an order in a more complicated exceptional component of type (4).
Examples of this are non totally definite quaternion algebras over extensions of $\Q$ of degree higher than $2$. Such an algebra shows up in the Wedderburn Artin decomposition of $\Q(Q_8 \times C_7)$, where $Q_8$ is the quaternion group of order $8$ and $C_7$ the cyclic group of order $7$. One of the components in the decomposition is $(\frac{-1,-1}{\Q(\xi_7)})$. In \cite{BCNS15}, the authors apply Vorono\"i's algorithm, while in \cite{JKdR16}, the authors consider actions on direct products of hyperbolic spaces. 

The exceptional components of type (3) were until now barely dealt with, except for the paper \cite{eiskiefvgel}, which we discuss later.
The goal of this paper is to develop a new method that deals with these components. This work is in fact a continuation of the work done in \cite{algebrapaper}. As explained above, in that paper, the authors use the action on hyperbolic spaces of dimension $2$ and $3$ and develop an algorithm, which takes as an input a discrete cofinite subgroup of $\PSL_2(\R)$ or $\PSL_2(\C)$ to return a finite-sided convex polyhedron that contains a Dirichlet fundamental domain of the discrete subgroup. Via this polyhedron a finite  generating set of a finite index subgroup of the given group is then computed. As stated above, there are already several algorithms that compute exact fundamental domains for groups acting properly discontinuously on hyperbolic space. The advantage of the algorithms in \cite{algebrapaper} and in this paper is that they are much simpler, as they do not need to use the complicated machinery of Poincar\'e's Polyhedron Theorem and do not need to deal with cusp points. In return the output will not be an exact fundamental domain for the input group, but only a polyhedron containing the fundamental domain. This polyhedron then yields a finite generating set of a finite index subgroup. As we are working with finite index subgroups, this is enough for our purpose. 

The algorithm developed in this paper may be applied to orders in non-totally definite quaternion algebras with centre $\Q$ or an imaginary quadratic extension of $\Q$ or to orders $\SL_2(\O)$ where $\O$ is an order in $\Q$ or an imaginary quadratic extension of $\Q$. 

In this paper, we generalize the DAFC algorithm to hyperbolic space of dimension $5$, in order to get a finite set of generators of finite index subgroups of $\SL_2((\frac{x,y}{\Z}))$. In  \cite[Theorem 7]{Wilker93}, it is shown that $\PSL_2((\frac{-1,-1}{\R}))$ acts properly discontinuously on hyperbolic space of dimension $5$. Unfortunately this action is not given directly by M\"obius transformations, but one has to pass via  the Poincar\'e extension, in contrast with  the classical case of hyperbolic spaces of dimension $2$ and $3$. Vahlen \cite{Vahlen02} introduced a special linear group of $2$-by-$2$ matrices over the real Clifford algebra $\Cliff_n(\R)$ of degree $n$. In \cite{Maass49} Maass also described this group. Recall that a Clifford algebra may also be seen as a twisted group algebra with a special cocycle. For details, we refer to \cite{AlbMaj}. They prove that it acts via M\"obius transformations on hyperbolic space of dimension $n+1$ and constitutes the group of orientation preserving isometries of $\Hy^5$. This generalizes the classical cases of Fuchsian and Kleinian groups to higher dimensions. In 1985, Ahlfors draws attention to this method in \cite{1985ahlfors} and denotes this group by $\SL_+(\Gamma_n(\R))$. More recent papers about generalizations of Kleinian groups to higher dimensions include \cite{kapovich07}, \cite{maclachlanetall} and \cite{waterman93}. In this paper we consider the discrete subgroup $\SL_+(\Gamma_n(\Z))$, which acts properly discontinuously on $\Hy^{n+1}$. In particular we consider the action of $\PSL_+(\Gamma_4(\Z))$ on $\Hy^5$ and generalize the DAFC to this case. In this way, we get an algorithm to compute a finite set of generators of finite index subgroups of $\SL_+(\Gamma_4(\Z))$. We then make use of the fact that $\SL_+(\Gamma_4(\Q))$ is isomorphic to $\SL_2((\frac{-1,-1}{\Q}))$ and use the concrete isomorphism given in \cite{ElsGrunMen} to get a finite generating set of finite index subgroups of $\SL_2((\frac{-1,-1}{\Z}))$. Concerning matrices over general quaternion algebras, i.e. $\SL_2((\frac{x,y}{\Z}))$ with $x$ and $y$ negative integers not necessarily equal to $-1$, we make use of the fact that over $\R$ all totally definite quaternion algebras are isomorphic. We hence embed $\SL_2((\frac{x,y}{\Z}))$ into $\SL_2((\frac{-1,-1}{\R}))$ and re-use the same technique. 

Note that in \cite{eiskiefvgel}, the authors showed that the exceptional components of type (3) can be reduced to the following three matrix rings over quaternion algebras:
\begin{itemize}
\item $M_2((\frac{-1,-1}{\Q}))$,
\item $M_2((\frac{-1,-3}{\Q}))$,
\item $M_2((\frac{-2,-5}{\Q}))$.
\end{itemize}
The surprising fact is that the three quaternion algebras appearing in those components have a norm Euclidean maximal order (see \cite[Theorem 2.1]{2012Fitzgerald}). This makes it possible to describe generators of $\SL_2(\O)$ with $\O$ a maximal order in these algebras by imitating the classical process, based on the Euclidean algorithm, in which generators for $\SL_2(\Z)$ are computed. For details see \cite[Proposition 4.1]{eiskiefvgel}. So for the pure purpose of units in group rings, there is a way of computing generators for the unit group of a maximal order in components of type (3) and the details may be found in \cite{eiskiefvgel}. 
Nevertheless this paper is still interesting, because the algorithm presented in this paper works for the unit group of orders of a $2$-by-$2$ matrix ring over some general totally definite quaternion algebra, i.e. that does not necessarily contain a norm Euclidean maximal order. Moreover, even in the context of units in group rings, one is often interested in the unit group of an order in a congruence subgroup of a matrix ring, and not the whole matrix ring. In the latter case, the classical computation, involving the Euclidean algorithm, is not working any more, as it requires the full matrix ring. Hence the interest of the algorithm presented in this paper. 

The outline of the paper is as follows. In section 2, we give the necessary background on hyperbolic spaces, Poincar\'e's Method, Quaternion algebras and Clifford algebras. In Section 3, we generalize the necessary results from \cite{algebrapaper} to higher dimensions. Finally in Section 4, we give the new algorithm. To stay consistent with \cite{algebrapaper}, we also work in this paper with the upper half-space and the ball model of hyperbolic space. In Section 5, we show a concrete method for computing generators of a subgroup of finite index in $\SL_2((\frac{-1,-1}{\Z}))$. This method is generalized in Section 6 to subgroups of finite index of $\SL_2((\frac{x,y}{\Z}))$, for $x$ and $y$ some general negative integers. To finish the paper, we give two examples in Section 7.

\section{Preliminaries}

In this section we recall necessary information about hyperbolic space, its orientation preserving isometries and Poincar\'e's Method. We further show how these concepts can be generalized to Clifford matrices. 

\subsection{Hyperbolic Space and M\"obius Transformations}
Standard references on hyperbolic geometry are \cite{beardon, bridson,
elstrodt,ratcliffe}. Let $\Hy^n$ (respectively $\BQ^n$) denote the upper half-space (respectively the ball) model of hyperbolic $n$-space.  Denote by $\mbox{Iso} (\Hy^{3})$ the group of isometries of $\Hy^{3}$. The group of
orientation preserving isometries is denoted by $\mbox{Iso}^{+}(\Hy^{3})$.

Let us first briefly consider the special case of dimension $2$ and $3$. Here $\Hy^2= \C= \lbrace x + yi \mid x \in \R, y \in \R^+ \rbrace$ and $\Hy^{3}$ can be identified with the subset of the classical real quaternion algebra $\h(\R)=(\frac{-1,-1}{\R})$
by identifying  $\Hy^{3}$ with the subset  $\{z+rj
\in \h(\R) \mid z \in \C, r \in \R^{+}\}\subseteq \h(\R)$. The
ball models $\BQ^2$ and $\BQ^{3}$ may be identified in the same way  with $\lbrace x+yi \in \C \mid x^2 + y^2 <1\rbrace$ and  
$\{z+rj\in \C +\R j \mid  |z|^2+r^2<1\}\subseteq
\h(\R)$, respectively.  
It is easy to see that $\PSL_2(\R)$ and $\PSL_2(\C)$ act on $\R$ and $\C$ respectively by M\"obius transformations, i.e. for $z \in \R$ or $z \in \C$
$$\begin{pmatrix}
a & b\\
c & d
\end{pmatrix}z =\frac{az+b}{cz+d} \in  \R \textrm{ or } \C \textrm{ respectively}.$$

\begin{remark}
By abusing the notation, we denote by $M = \begin{pmatrix} a & b \\ c & d \end{pmatrix}$ an element of $\SL_2(\C)$ as well as its projection to $\PSL_2(\C)$.
\end{remark}

By the so-called Poincar\'e extension, this action can be extended to $\Hy^2$ and $\Hy^2$ respectively. 
It is then well known (see for instance \cite{beardon}) that $\mbox{Iso}^{+}(\Hy^{2})$ is isomorphic to
$\PSL_2(\R)$ and $\mbox{Iso}^{+}(\Hy^{3})$ is isomorphic to
$\PSL_2(\C)$. The action is again by M\"obius transformations. More concretely, the action of $\PSL_2(\R)$ (or $\PSL_2(\C)$)  on  $\Hy^{2}$ (or $\Hy^{3}$) is given
by
\begin{eqnarray*}
\begin{pmatrix}
a & b\\
c & d
\end{pmatrix}z = (az+b)(cz+d)^{-1},
\end{eqnarray*}
where $z \in \Hy^2$ (or $\Hy^3$) and $(az+b)(cz+d)^{-1}$  is evaluated in $\C$ (or in the algebra $\h(\R)$). Explicitly, if $M = \begin{pmatrix}
a & b\\
c & d
\end{pmatrix}$ then 
\begin{eqnarray}
M (x+yi)=\frac{(ax+b)(cx+d)+a
cy^2}{(cx+d)^{2}+c^{2}y^{2}} + \frac{y}{(cz+d)^{2}+c^{2}y^{2}}i, \label{formule1} \\
M (z+rj)=\frac{(az+b)(\overline{c}\overline{z}+\overline{d})+a
\overline{c}r^{2}}{|cz+d|^{2}+|c|^{2}r^{2}} + \frac{r}{|cz+d|^{2}+|c|^{2}r^{2}}j. \label{formule2}
\end{eqnarray}
This action may be extended to $\widehat{\Hy}^2=\Hy^2 \cup \partial \Hy^2 \cup \lbrace \infty \rbrace$ (or $\widehat{\Hy}^3=\Hy^3 \cup \partial \Hy^3 \cup \lbrace \infty \rbrace$).

\begin{remark}\label{abuse}
Note that throughout the article we write $a=a(M)$, $b=b(M)$, $c=c(M)$ and $d=d(M)$, for the entries of $M =\begin{pmatrix}
a & b\\
c & d
\end{pmatrix}$, when it is necessary to stress the dependence of the entries on the matrix $M$.
\end{remark}

The orientation preserving isometries of the ball model may be also explicitly described in dimensions $2$ and $3$. We will do this for dimension $3$, dimension $2$ being done in a similar fashion. 
Let $a=a_0+a_1i+a_2j+a_3k \in \h(\R)$. We define the following automorphism and anti-automorphisms:
\begin{itemize}
\item $a'=a_0-a_1i-a_2j+a_3k$,
\item $a^{*}=a_0+a_1i+a_2j-a_3k$,
\item $\overline{a}=a_0-a_1i-a_2j-a_3k$. 
\end{itemize}
Let 
\begin{equation*}
\SB_{2} (\h(\R) )=\left\{
\left(
\begin{array}{cc}
a & b \\ c & d
\end{array}
\right) \in M_2(\h(\R)  ) \mid  d=a^{\prime},\;  b=c^{\prime},\;  a\overline{a}-c\overline{c}=1,\; a\overline{c} \in \C +\R j\right\}.
\end{equation*}
The following proposition describes the link between the upper-half space and the ball model.

\begin{proposition}\label{propels}\cite[Proposition I.2.3]{elstrodt}
\begin{enumerate}[(i)]
\item For $P \in \Hy^3$, the quaternion $-jP+1$ is invertible in $\h(\R)$ and the  map $\eta_0 :
\Hy^3\longrightarrow \BQ^3$, given by
$ \eta_0(P)=(P-j)(-jP+1)^{-1}$, is an isometry. More precisely $\eta_0=\mu\pi$, where $\pi$ is the reflection in the boundary of $\Hy^3$ and $\mu$ is the reflection in the Euclidean sphere with centre $j$ and radius $\sqrt{2}$.
\item Let $\Lambda=\frac{1}{\sqrt{2}}\left(
\begin{array}{ll}
1 & j \\ j & 1
\end{array}
\right) \in M_{2} (\h(\R) )$. The map
 $\Psi : \SL_2(\C) \rightarrow SB_{2}(\h(\R) )$ given by  $\Psi (M ) =
\overline{\Lambda}M \Lambda $
is a group   isomorphism.
\item For $u \in \BQ^3$ and $f=\begin{pmatrix} a & c' \\ c & a' \end{pmatrix} \in SB_2(\h(\R))$ the quaternion $cu+a'$ is invertible in $\h(\R)$ and the transformations $f:\BQ^3 \longrightarrow \BQ^3$, defined by $f(u)=(au+c')(cu+a')^{-1}$ are isometries of $\BQ^3$ and define an action of $SB_2(\h(\R))$ on $\BQ^3$. Again this action may be extended to the closure of $\BQ^3$, which we denote by $\overline{\BQ^3}$.
\item The group $\mbox{Iso}^{+}(\BQ^3 )$ is isomorphic to $SB_2(\h(\R))/ \lbrace 1,-1 \rbrace$.
\item The map $\eta_0$ is equivariant with respect to $\Psi$, that is $\eta_{0} (MP)=\psi (M) \eta_{0}(P)$, for $P\in \Hy^{3}$ and $M\in \SL_2(\C)$ .
\end{enumerate}
\end{proposition}

\subsection{Poincar\'e's Method}\label{prelimpoin}

Recall that a group $\Gamma$ is said to act properly discontinuously on a locally compact space $X$ if for every compact subset $K$ of $X$, $K \cap \gamma(K) \neq \emptyset$ for only finitely many $\gamma \in \Gamma$. If $X$ is a proper metric space, then a group $\Gamma \leq \textrm{Iso}(X)$ acts properly discontinuously on $X$ if and only if $\Gamma$ is a discrete subgroup of $\textrm{Iso}(X)$. For more details on this, see \cite[Theorem 5.3.5]{ratcliffe}. The following definitions are taken from \cite[Section 6.6]{ratcliffe}. A subset $F$ of a metric space $X$ is a fundamental set for a
group $\Gamma$ of isometries of X if and only if $F$ contains exactly one point from
each $\Gamma$-orbit in $X$. A subset $R$ of a metric space $X$ is a fundamental region for a group $\Gamma$ acting properly discontinuously by isometries on $X$ if and only if \begin{itemize}
\item $R$ is a an open set in $X$,
\item the members of $\lbrace \gamma(R) \mid \gamma \in \Gamma \rbrace$ are mutually disjoint and
\item $X = \cup_{\gamma \in \Gamma} \gamma(\overline{R})$.
\end{itemize}
A fundamental domain $\F$ is a connected fundamental region. If moreover $X$ is a space of constant curvature, we call $\F$ a convex fundamental polyhedron for $\Gamma$, if $\F$ is a convex polyhedron whose interior is a locally finite fundamental domain for $\Gamma$. Note that by \cite[Corollary 6.7.1]{ratcliffe}, the boundary of every convex locally finite fundamental domain is a null-set.

In this paper we work with the Dirichlet fundamental domain. We recall its construction and how it can be used to give a presentation for the considered groups, the so called Poincar\'e method (for details see for example \cite{beardon} and \cite{ratcliffe}). 
Let $\Gamma$ be a group acting properly discontinuously on $\Hy^n$, $\gamma \in \Gamma$ and $P \in \Hy^n$ a point which is not fixed by any $\gamma \in \Gamma \setminus \lbrace 1 \rbrace$. Then let
\begin{equation}\label{defbis}
D_{\gamma}(P)=\{z \in \Hy^n \mid \rho(P,z) < \rho (z,\gamma (P))\},
\end{equation}
the half-space containing $P$, where $\rho$ denotes the hyperbolic distance in $\Hy^n$. 
The boundary of $D_{\gamma}(P)$ is the bisector of $P$ and $\gamma(P)$, that is the set $\{z\in \Hy^n \mid \rho(P,z)= \rho (z,\gamma (P))\}$. We denote it by $\Sigma_{\gamma}(P)$. If $P \in \Hy^n$ is a point which is not fixed by any non-trivial element of $\Gamma$, then
\begin{equation*}
\F = \cap _{1 \neq \gamma \in \Gamma} D_{\gamma}(P)
\end{equation*}
is known as a Dirichlet fundamental polyhedron of $\Gamma$ with centre $P$. By \cite[Theorem 6.6.13]{ratcliffe}, the Dirichlet fundamental domain is a locally finite fundamental domain. Moreover $\F$ is convex as $\F$ is the intersection of open half-spaces of $\Hy^n$. Hence by \cite[Theorem 6.7.1]{ratcliffe}, the closure of the Dirichlet fundamental domain is a fundamental polyhedron. If $\Gamma_P$, the stabilizer of $P$ in $\Gamma$, is not trivial and if $\overline{\F_P}$ is a fundamental polyhedron for the group $\Gamma_P$, then the closure of
\begin{equation}\label{fundwithstab}
\F = \F_P \cap ( \cap_{\gamma \in \Gamma \setminus \Gamma_P} D_{\gamma}(P))
\end{equation}
is a fundamental polyhedron of $\Gamma$. A proof of this can be found for example in \cite[Proof of Proposition 3.2]{jesetall}.
Once a fundamental polyhedron for $\Gamma$ is constructed, Poincar\'e's Method gives a way of describing generators for the group $\Gamma$.

\begin{theorem}\label{Poincare}\cite[Theorem 6.8.3]{ratcliffe}
Let $\F$ be a convex fundamental polyhedron for a group $\Gamma$ acting properly discontinuously on a metric space $X$ of constant curvature. 
Then $\Gamma$ is generated by the set
\begin{equation*}
\lbrace \gamma \in \Gamma \mid \F \cap \gamma(\F) \textrm{ is a side of } \F \rbrace.
\end{equation*}
The elements $\gamma$ of the generating set are called side-pairing transformations.
\end{theorem}

\begin{remark} Note that Theorem~\ref{Poincare} is not Poincar\'e's Polyhedral Theorem in its full generality, but rather a corollary of it. 
\end{remark}

\subsection{The group \texorpdfstring{$\PSL_2((\frac{x,y}{\Z}))$}{LG}}

In this subsection we consider quaternion algebras over the field $\Q$. A $\Q$-algebra is said to be a quaternion algebra over $\Q$ if there exist $x,y \in \Q^*$ and a $\Q$-basis $\lbrace 1,i,j,k \rbrace$, such that 
$$i^2=x, j^2=y \textrm{ and } ij=-ji=k.$$
This algebra is denoted $(\frac{x,y}{\Q})$. Clearly $(\frac{-1,-1}{\Q})$ is the classical Hamilton quaternion algebra over $\Q$. We are interested in quaternion algebras that are skew fields, i.e. $a_0^2-xa_1^2-ya_2^2-xya_3^2=0$ for some $(a_0,a_1,a_2,a_3) \in \Q^4$ if and only if $a_0=a_1=a_2=a_3=0$. As described in the Introduction, we are interested in the elements of $M_2((\frac{x,y}{\Z}))$ of reduced norm $1$ over $\Q$. We recall the notion of reduced norm. Let $A$ be a finite dimensional central simple algebra over $K$. Let $E$ be a splitting field of $A$, i.e. $E \otimes_K A \cong M_{n}(E)$ is a full matrix ring over $E$. The \emph{reduced norm} of $a \in A$ is defined as \[\operatorname{RNr}_{A/K}(a) = \det(1 \otimes a). \] 
Note that $\operatorname{RNr}_{A/K}(\cdot)$ is a multiplicative map, $\operatorname{RNr}_{A/K}(A) \subseteq K$ and $\operatorname{RNr}_{A/K}(a)$ do only depend on $K$ and $a \in A$ (and not on the chosen splitting field $E$ and isomorphism $E \otimes_K A \cong M_{n}(E)$), see \cite[page 51]{ericangelbook}. For a subring $R$ of $A$, put
 \[ \SL_1(R) = \{\ a \in \U(R)\ |\ \operatorname{RNr}_{A/K}(a) = 1\ \}, \]
 which is a (multiplicative) group. Here $A = M_2((\frac{x,y}{\Q}))$ and $R= M_2((\frac{x,y}{\Z}))$ and we write $\SL_1(A)= \SL_2((\frac{x,y}{\Q}))$ and $\SL_1((\frac{x,y}{\Z})) = \SL_2((\frac{x,y}{\Z}))$.

As the reduced norm is not handy to work with, we recall the notion of Dieudonn\'e determinant. 
Let $\begin{pmatrix} a & b \\ c & d \end{pmatrix}$ be a matrix in $M_2((\frac{x,y}{\Q}))$. We first define the pseudo-determinant $\sigma$ as follows:
\begin{equation}\label{defsigma}
\begin{aligned}
\sigma = \begin{cases}
cac^{-1}d-cb & \textrm{ when } c \neq 0, \\
bdb^{-1}a & \textrm{ when } c=0,b \neq 0,\\
(d-a)a(d-a)^{-1}d & \textrm{ when } b=c=0, a \neq d,\\
a\overline{a} & \textrm{ when } b=c=0, a = d.
\end{cases}\end{aligned}\end{equation}
Based on this quantity $\sigma$, the Dieudonn\'e determinant of this matrix is defined as
$$ \Delta = \vert \sigma \vert = \sqrt{\vert a \vert^2\vert d \vert^2 + \vert b \vert^2\vert c \vert^2 - 2Re\left[a\overline{c}d\overline{b}\right]}.$$
Note that $\Delta$ is positive and real. Moreover it may be proven (see \cite[Theorem~1, page~146]{draxl}) that $\Delta^2$ equals the reduced norm. For more information on reduced norms and the Dieudonn\'e determinant, we refer an interested reader to \cite{draxl}. 
As a consequence we obtain the following lemma. 

\begin{lemma}
A matrix $M \in M_2(\h(\R))$ is invertible if and only if $\Delta \neq 0$. Furthermore the inverse is given by 
\begin{equation}\label{inversdeltamatrix}
\begin{aligned}
M^{-1} = \begin{pmatrix} a^{-1} & -\sigma^{-1}b \\ 0 & d^{-1} \end{pmatrix}, & \textrm{ when } c=0,\\
M^{-1} = \begin{pmatrix} c^{-1}d\sigma^{-1}c & -a^{-1}b\sigma^{-1}cac^{-1} \\ -\sigma^{-1}c & \sigma^{-1}cac^{-1} \end{pmatrix}, & \textrm{ when } c\neq 0.
\end{aligned}\end{equation}  
\end{lemma}

Consider a matrix $M \in M_2((\frac{x,y}{\Z}))$. It is easy to see that if $M$ is a unit in $M_2((\frac{x,y}{\Z}))$, then $\Delta = 1$. On the other hand, if $\Delta=1$, then $\sigma^{-1}=\overline{\sigma}$. With some computations, one may prove that in that case all the four entries of $M^{-1}$ are in $(\frac{x,y}{\Z})$. We hence obtain the following lemma.
\begin{lemma}\label{unitinMHZ}
A matrix $M \in M_2((\frac{x,y}{\Z}))$ is a unit in $M_2((\frac{x,y}{\Z}))$ if and only if $\Delta=1$.
\end{lemma}
Similarly to the classical case, we define the groups
$$\SL_2\left(\left(\frac{x,y}{\R}\right)\right) = \left\lbrace \begin{pmatrix} a & b \\ c & d \end{pmatrix} \mid a,b,c,d \in \left(\frac{x,y}{\R}\right), \Delta = 1 \right\rbrace,$$
$$\SL_2\left(\left(\frac{x,y}{\Z}\right)\right) = \left\lbrace \begin{pmatrix} a & b \\ c & d \end{pmatrix} \mid a,b,c,d \in \left(\frac{x,y}{\Z}\right), \Delta = 1 \right\rbrace.$$
By passing to the projective group 
$$\PSL_2\left(\left(\frac{x,y}{\R}\right)\right) = \PSL_2\left(\left(\frac{x,y}{\R}\right)\right)/\lbrace \pm I \rbrace,$$
we get the following isomorphism (see \cite{kellerhals03} for details)
$$\PSL_2\left(\left(\frac{x,y}{\R}\right)\right) \cong \Isom^{+}(\Hy^5).$$

\subsection{Clifford M\"obius Transformations}\label{Cliffordsection}

The Clifford algebra $\Cliff_{n}(\R)$ is the associative algebra over the real numbers generated by $n-1$ elements $i_1, \ldots, i_{n-1}$ subject to the relations $i_hi_k = -i_ki_h$ for $h \neq k$ and $i_h^2=-1$ for $1 \leq h \leq n-1$ and no other relations. The Clifford vectors are the elements of $\Cliff_n(\R)$ of the special form $\alpha=\alpha_0+\alpha_1i_1+\alpha_2i_2+ \ldots \alpha_{n-1}i_{n-1}$. The space of all the Clifford vectors is a vector subspace of $\Cliff_n(\R)$ and we denote it by $\V^n(\R)$. The main conjugation consists of replacing $i_h$ by $-i_h$. It is an automorphism of $\Cliff_n(\R)$. We denote the conjugate of $\alpha$ by $\alpha'$. The conjugation denoted by $\alpha \mapsto \alpha^{*}$ consists in reversing the order of the factors in each product $i_{h_1}\ldots i_{h_m}$. This defines an anti-automorphism on $\Cliff_n(\R)$ in the sense that for $\alpha,\beta \in \Cliff_n(\R)$, $(\alpha \beta)^*=\beta^*\alpha^*$. Finally, combining both conjugations, we get another anti-automorphism denoted by $\alpha \mapsto \overline{\alpha}=\alpha'^{*}=\left(\alpha^{*}\right)'$. If $\alpha \in \V^n(\R)$, then $\alpha^* = \alpha $ and $\alpha' = \overline{\alpha}$. If $\alpha \in \Cliff_n(\R)$, write $\alpha = \sum \alpha_I I$, where $I=i_{j_1} \ldots i_{j_r}$ with $i_{j_i} \in \lbrace i_1, \ldots , i_{n-1} \rbrace$. Define $\vert \alpha \vert^2=\sum \alpha_I^2$. If $\alpha \in \V^n(\R)$, then $\alpha\overline{\alpha}=\vert \alpha \vert^2$. This also leads to the fact that every non-zero vector $\alpha$ is invertible and its inverse is given by $\alpha^{-1}=\overline{\alpha}\vert \alpha \vert^{-2}$. In order to work with matrices we want to be able to multiply vectors and hence to work in a multiplicative group. That is why we define the Clifford group $\Gamma_n(\R)$ whose elements are all the possible products of invertible vectors. These products are of course invertible and therefore $\Gamma_n(\R)$ is a group. 

We consider matrices of the form $M=\begin{pmatrix} \alpha & \beta \\ \gamma & \delta \end{pmatrix}$ with $\alpha,\beta,\gamma,\delta \in \Gamma_n(\R) \cup \lbrace 0 \rbrace$ and we would like $M$ to define a bijective mapping from $\overline{\V}^n(\R)$ to $\overline{\V}^n(\R)$, where $\overline{\V}^n(\R) = \V^n(\R) \cup \lbrace \infty \rbrace$.

\begin{definition}\cite[Definition 2.1]{1985ahlfors}\label{defcliffmoeb}
\begin{eqnarray*}
\GL(\Gamma_n(\R)) & = & \left \lbrace \begin{pmatrix} \alpha & \beta \\ \gamma & \delta \end{pmatrix} \mid \alpha,\beta,\gamma,\delta \in \Gamma_n(\R) \cup \lbrace 0 \rbrace; \alpha\delta^*-\beta\gamma^* \in \R\setminus \lbrace 0\rbrace; \right.\\
& & \left. \alpha\beta^*, \gamma\delta^*, \gamma^*\alpha, \delta^*\beta \in \V^n(\R) \right \rbrace,\\
\SL_+(\Gamma_n(\R)) & = & \left \lbrace \begin{pmatrix} \alpha & \beta \\ \gamma & \delta \end{pmatrix} \in \GL(\Gamma_n(\R))  \mid \alpha\delta^*-\beta\gamma^* =  1 \right \rbrace.
\end{eqnarray*}
\end{definition}

\begin{remark}\label{rationalClifford}
Note that all the above definitions make sense if we replace the field $\R$ by $\Q$. Hence we get $\Cliff_n(\Q)$, $\V^n(\Q)$, $\Gamma_n(\Q)$ and $\SL_+(\Gamma_n(\Q))$. 
\end{remark}

In \cite{1985ahlfors}, it is shown that the sets from Definition~\ref{defcliffmoeb} are well defined multiplicative groups. In \cite{ElsGrunMen} a different definition of $\SL_+(\Gamma_n(\R))$ is given in Definition 3.1. It is then shown in \cite[Theorem 3.7]{ElsGrunMen} that both definitions are equivalent. The next theorem is taken from \cite{1985ahlfors}.

\begin{theorem}\cite[Theorem A]{1985ahlfors}
The matrix $M = \begin{pmatrix} \alpha & \beta \\ \gamma & \delta \end{pmatrix}$ with $\alpha,\beta,\gamma,\delta \in \Gamma_n(\R) \cup \lbrace 0 \rbrace$ induces bijective mappings $\overline{\V}^n \rightarrow \overline{\V}^n$  and $\overline{\V}^{n+1} \rightarrow \overline{\V}^{n+1} $ if and only if $M \in \GL(\Gamma_n(\R))$.
\end{theorem}

Let
\begin{equation}\label{projection}
\Pi: \SL_+(\Gamma_n(\R)) \rightarrow \PSL_+(\Gamma_n(\R))
\end{equation}
be the natural projection (see \cite[(5.6)]{ElsGrunMen}). By passing to the projective group $\Pi(\SL_+(\Gamma_n(\R))=\PSL_+(\Gamma_n(\R))$, we get the following theorem.

\begin{theorem}\cite[Theorem B]{1985ahlfors}\label{isomH5}
The group $\PSL_+(\Gamma_n(\R))$ is isomorphic to the groups $M_+(\overline{\R}^n)$ and $M_+(\Hy^{n+1})$ of orientation preserving M\"obius transformations.
\end{theorem}

An element $z \in \Hy^{n+1}$ is given by $z=z_0+z_1i_1 + \ldots +z_{n}i_{n}$ with $z_{n} \geq 0$ and a matrix $A \in \PSL_+(\Gamma_n(\R))$ will act on $z$ in the following way
$$\begin{pmatrix} \alpha & \beta \\ \gamma & \delta \end{pmatrix} z = (\alpha z+\beta)(\gamma z+\delta)^{-1},$$
where the latter is calculated in the algebra $\Cliff_{n+1}(\R)$. The conditions on $\PSL_+(\Gamma_n(\R))$ given in Definition~\ref{defcliffmoeb} guarantee that $(\gamma z+\delta)^{-1}$ is well defined in $\Cliff_{n+1}(\R)$ and that $(\alpha z+\beta)(\gamma z+\delta)^{-1}$ is an element of $\Hy^{n+1}$ (for details see \cite[Theorem A]{1985ahlfors}). Moreover (\ref{formule1}) and (\ref{formule2}) can also be generalized. Indeed write $z$ as $z=y + ri_n$, where $y \in \V^{n}(\R)$ and $r \in \R_+$. Then, it is easy to compute that
\begin{eqnarray} \label{expressionimage}
(\alpha z+\beta)(\gamma z+\delta)^{-1} 
& = & \frac{\alpha\overline{\gamma}\vert z \vert^2 + \beta\overline{\delta} + \alpha y\overline{\delta} + \beta \overline{y} \overline{\gamma}}{\vert \gamma z +\delta \vert^2} + \frac{ri_n}{\vert \gamma z +\delta \vert^2}.
\end{eqnarray}

The equivalent definition of $\SL_+(\Gamma_n(\R))$ of \cite[Definition 3.1]{ElsGrunMen} ensures that (\ref{expressionimage}) is a vector of $\Hy^{n+1}$. 
Moreover, from Theorem~\ref{isomH5}, we get the following corollary.

\begin{corollary}
The groups $\PSL_2(\h(\R))$ and $\PSL_+(\Gamma_4(\R))$ are isomorphic.
\end{corollary}

We cite two more lemmas from \cite{1985ahlfors}, which will be useful later.

\begin{lemma}\label{simulVn}\cite[Lemma 1.4]{1985ahlfors}
If $\alpha, \beta \in \Gamma_n(\R)$, then $\alpha\beta^{-1}$ and $\alpha^*\beta$ are simultaneously in $\V^n(\R)$. 
\end{lemma}

\begin{lemma}\label{inverscliffmat}\cite{1985ahlfors}
If $M = \begin{pmatrix} \alpha & \beta \\ \gamma & \delta \end{pmatrix} \in \SL_+(\Gamma_n(\R))$, then $M^{-1} = \begin{pmatrix} \delta^* & -\beta^* \\ -\gamma^* & \alpha^* \end{pmatrix}$.
\end{lemma}

Let $\Cliff_n(\Z)$ denote the module of integral combinations of the elements $i_1, \ldots, i_{n-1}$ and $\Gamma_n(\Z)$ the monoid of products of vectors in $\Cliff_n(\Z)$, that are invertible in $\Cliff_n(\R)$, i.e. $$\Gamma_n(\Z)=\Cliff_n(\Z) \cap \Gamma_n(\R).$$ 
Then we define $\SL_+(\Gamma_n(\Z))$ as the subgroup of $\SL_+(\Gamma_n(\R))$ of all matrices with entries in $\Gamma_n(\Z) \cup \lbrace 0 \rbrace$. More precisely
$$\SL_+(\Gamma_n(\Z)) = \SL_+(\Gamma_n(\R)) \cap M_2(\Gamma_n(\Z) \cup \lbrace 0 \rbrace).$$ The projective group $\PSL_+(\Gamma_n(\Z))$ is again defined via the natural projection \eqref{projection}. The group $\SL_+(\Gamma_n(\Z))$ can bee seen as the group of integer points of an algebraic subgroup of $\GL_{2n+1}(\C)$ and by a  result of Borel and Harish-Chandra on arithmetic groups being finitely presented \cite{BorHarCha}, we get the following theorem (for details see \cite[Theorem 3]{maclachlanetall} or \cite[Proposition 2.3]{ElsGrunMen2}).

\begin{theorem}\label{finitelygen}
$\PSL_+(\Gamma_n(\Z))$ is a finitely generated arithmetic group of finite covolume.
\end{theorem}

$\PSL_+(\Gamma_n(\Z))$ is clearly a discrete subgroup of $\PSL_+(\Gamma_n(\R))$. Thus, $\PSL_+(\Gamma_n(\Z))$ acts on $\Hy^{n+1}$ properly discontinuously with quotient of finite volume. The following corollary follows immediately from \cite{maclachlanetall} and \cite[Proposition 5.6]{bowditch}.

\begin{corollary}\label{geomfinite}
The Dirichlet fundamental polyhedron of $\PSL_+(\Gamma_n(\Z))$ has finitely many sides.
\end{corollary}

\section{Generalizing the general theory to Clifford matrices}

By Theorem~\ref{isomH5}, every discrete subgroup $G$ of $\PSL_+(\Gamma_n(\R))$ acts properly discontinuously on the hyperbolic space $\Hy^{n+1}$. 
For the purpose of the algorithm we will have to switch from the upper half-space to the ball model and vice-versa. Therefore we first show that Proposition~\ref{propels} may be generalized to higher dimensions in the context of Clifford matrices. As in Proposition~\ref{propels}, let $\eta_0=\mu\pi$, where $\pi$ is the reflection in the boundary of $\Hy^{n+1}$ and $\mu$ is the reflection in the Euclidean sphere with centre $i_n$ and radius $\sqrt{2}$. By definition $\eta_0$ is an isometry between the two models (for details, see \cite{ratcliffe}). Let $\Lambda=\frac{1}{\sqrt{2}}\left(
\begin{array}{ll}
1 & i_n \\ i_n & 1
\end{array}
\right) \in M_{2} (\Cliff_{n+1}(\R) )$. It is easy to verify that the M\"obius transformation $\overline{\Lambda}(z)=(z-i_n)(-i_nz+1)^{-1}$ describes exactly the map $\eta_0$. Hence we may generalize Proposition~\ref{propels} in the following way. See also \cite[Proposition 5.6]{ElsGrunMen}.

\begin{proposition}\label{propels2}
\begin{enumerate}[(i)]
\item For $z \in \Hy^{n+1}$, the Clifford element $-i_nz+1$ is invertible in $\Cliff_{n+1}(\R)$ and the  map from
$\Hy^{n+1}$ to $\BQ^{n+1}$, mapping $z$ on $(z-i_n)(-i_nz+1)^{-1}$, is an isometry. 
\item Let $\Lambda=\frac{1}{\sqrt{2}}\left(
\begin{array}{ll}
1 & i_n \\ i_n & 1
\end{array}
\right) \in M_{2} (\Cliff_{n+1}(\R) )$. The map
 $\Psi : \SL_+(\Gamma_n(\R)) \rightarrow \SL_+(\Gamma_{n+1}(\R))$ given by  $\Psi(M) =
\overline{\Lambda}M \Lambda $
is a group  embedding.
\item The map $\eta_0$ is equivariant with respect to $\Psi$, that is $\eta_{0} (Mz)=\psi (M) \eta_{0}(z)$, for $z\in \Hy^{n+1}$ and $M\in \SL_+(\Gamma_n(\R))$ .
\end{enumerate}
\end{proposition}

\begin{proof}
We first prove item (i). Therefore we rewrite $-i_nz+1$ as $-i_n(z+i_n)$. As $i_n$ is invertible in $\Cliff_{n+1}(\R)$ and $z+i_n$ is a non-zero vector in $\V^{n+1}(\R)$, the latter is invertible in $\Cliff_{n+1}(\R)$. The rest of (i) follows from the discussion above. 

To prove (ii), we first note that $\Lambda \in \SL_+(\Gamma_{n+1}(\R))$. As every $M \in \SL_+(\Gamma_n(\R))$ can be seen as en element of $\SL_+(\Gamma_{n+1}(\R))$ and the latter is a group, it is clear that $\Psi(M) \in \SL_+(\Gamma_{n+1}(\R))$. 

Finally (iii) follows from the classical case.
\end{proof}

As in the classical case, one may give a concrete expression for $\Psi(M)$ for $M \in \SL_+(\Gamma_n(\R))$. This then gives an expression of the isometric sphere in the ball model $\BQ^{n+1}$, which allows to prove that in the ball model $\BQ^{n+1}$, the bisector $\Sigma_{\Psi(M)^{-1}}(0)$ associated to an element $M \in \SL_+(\Gamma_n(\R))$ equals the isometric sphere of $\Psi(M)$. This is summarized in the following lemma. We recall that the isometric sphere of a M\"obius transformation is the Euclidean sphere on which this transformation acts as a Euclidean isometry.  
 
\begin{lemma}\label{defpsigamma}
Let $M=\begin{pmatrix} \alpha & \beta \\ \gamma & \delta \end{pmatrix} \in \SL_+(\Gamma_n(\R))$ and let $\Psi(M)$ be defined as in Proposition~\ref{propels2}. Then we have the following.
\begin{enumerate}[(i)]
\item $\Psi(M)= \begin{pmatrix} \alpha+\delta' + (\beta-\gamma')i_n & \beta+\gamma' +(\alpha-\delta')i_n \\ \gamma+\beta' + (\delta-\alpha')i_n & \alpha' + \delta + (\gamma-\beta')i_n \end{pmatrix} = \begin{pmatrix} A & C' \\ C & A' \end{pmatrix}$,\\ with $A, C \in \SL_+(\Gamma_{n+1}(\R))$ such that $\vert A \vert ^2 - \vert C \vert^2=1$ and $A\overline{C} \in \V^{n+1}(\R)$.
\item If $C \neq 0$, the isometric sphere of $\Psi(M)$ in $\BQ^{n+1}$ is the sphere $S(C^{-1}A', \frac{1}{\vert C \vert})$ with centre $C^{-1}A'$ and radius $\frac{1}{\vert C \vert}$.
\item If $C \neq 0$, the point $\Psi(M)^{-1}(0)$ and the centre $P_{\Psi(M)}$ of the isometric sphere of $\Psi(M)$ are inverse points with respect to the unit sphere $\Sp^{n+1}= \partial \BQ^{n+1}$.
\item If $C \neq 0$, the bisector $\Sigma_{\Psi(M)^{-1}}(0)$ equals the isometric sphere of $\Psi(M)$.
\end{enumerate}
\end{lemma}

\begin{proof}
The first part of (i) follows from mere computations. Then by the definition of $\SL_+(\Gamma_{n+1}(\R))$, $1=AA'^* - CC'^* = \vert A \vert^2 - \vert C \vert^2$ and $AC'^* = A\overline{C} \in \V^{n+1}(\R)$.

By the proof of \cite[Theorem 4.4.4]{ratcliffe}, the centre of the isometric sphere of $\Psi(M)$ is given by $\Psi(M)^{-1}(\infty)$. By Lemma~ \ref{defpsigamma} and Lemma~\ref{inverscliffmat}, the latter equals $-\overline{A}(C^*)^{-1}$. By the definition of $\SL_+(\Gamma_{n+1}(\R))$, $CA'^* \in \V^{n+1}(\R)$ and hence  also $A'C^* =   (CA'^*)^*$ is a vector. By Lemma~\ref{simulVn}, $-\overline{A}(C^*)^{-1}$ is also a vector and hence it equals $-(\overline{A}(C^*)^{-1})^ *= C^ {-1}A'$. The square of the radius of the isometric sphere is then simply given by $\vert C^ {-1}A' \vert ^2 - 1 = \frac{1}{\vert C \vert^2}$. 

We now prove item (iii). By Lemma~\ref{inverscliffmat}, it is easy ot compute that $\Psi(M)^ {-1}(0) = -\overline{C}(A^*)^{-1}$. Let us denote the centre of the isometric sphere of $\Psi(M)$ by $P_{\Psi(M)}$. Then $P_{\Psi(M)} \cdot \Psi(M)^ {-1}(0) = \vert C \vert^ {-2} \vert A \vert ^ 2$. Thus $P_{\Psi(M)}$, $\Psi(M)^ {-1}(0)$ and $0$ are colinear in $\R^{n+1}$. Moreover $\vert P_{\Psi(M)} \vert \vert \Psi(M)^{-1}(0) \vert = 1$. This proves item (iii).

Finally to prove item (iv), first note that, by definition, the isometric sphere of $\Psi(M)$ is orthogonal to the unit sphere $\Sp^{n+1}= \partial \BQ^{n+1}$. Denote by $\Theta$ the intersection of the ray through $0$ and $P_{\Psi(M)}$ with the isometric sphere of $\Psi(M)$ and denote by $d$ the hyperbolic metric in $\BQ^{n+1}$. By \cite[Theorem 4.5.1]{ratcliffe}, it is easy to calculate that, for $z \in \BQ^{n+1}$,
$$d(0,z) = ln\left(\frac{1+\vert z \vert}{1-\vert z \vert}\right).$$
Thus $d(0,\Theta) = ln(\frac{\vert C \vert + \vert A \vert -1}{\vert C \vert - \vert A \vert+1})$ and $d(0,\Psi(M)^{-1}(0))=ln(\frac{\vert A \vert + \vert C \vert}{\vert A \vert - \vert C \vert})$. Hence $2d(0,\Theta) = d(0,\Psi(M)^{-1}(0))$, which proves that $\Theta$ is the midpoint of $0$ and $\psi(M)^{-1}(0)$ and thus the isometric sphere equals the bisector $\Sigma_{\Psi(M)^{-1}}(0)$.
\end{proof}

Finally we need one last lemma. For this, define 
\begin{eqnarray*}
\SU_+(\Gamma_n(\R)) & = & \SL_+(\Gamma_n(\R)) \cap \left \lbrace \begin{pmatrix} \alpha & \gamma \\ -\gamma' & \alpha' \end{pmatrix} \mid \alpha,\gamma \in \Gamma_n(\R) \cup \lbrace 0 \rbrace  \right \rbrace,\\
\PSU_+(\Gamma_n(\R)) & = & \SU_+(\Gamma_n(\R)) / \lbrace \pm I \rbrace.
\end{eqnarray*}
We also define the norm of an element $M = \begin{pmatrix} \alpha & \beta \\ \gamma & \delta \end{pmatrix} \in \SL_+(\Gamma_n(\R))$ as $\vert \alpha \vert^2 + \vert \beta \vert^ 2 + \vert \gamma \vert^2 +\vert \delta \vert^2$ and we denote it by $\Vert M \Vert^2$. The following proposition may also partly be found in \cite[Proposition 2.4]{ElsGrunMen3}.

\begin{lemma}\label{su}
Let $M  \in \SL_+(\Gamma_n(\R))$. Then the following are equivalent.
\begin{enumerate}[(i)]
\item $M \in \SU_+(\Gamma_n(\R))$.
\item $\Vert M \Vert^ 2= 2$.
\item $M$ fixes the point $i_n \in \Hy^{n+1}$.
\end{enumerate}
\end{lemma}

\begin{proof}
We show $(i) \Rightarrow (iii) \Rightarrow (ii) \Rightarrow (i)$. So suppose $M \in \SU_+(\Gamma_n(\R))$. Then by (\ref{expressionimage}), 
$M(i_n) = \frac{\alpha\overline{\gamma} - \gamma'\alpha^*}{\vert \gamma \vert^2 + \vert \alpha \vert^2} + \frac{i_n}{\vert \gamma \vert^2 + \vert \alpha \vert^2}.$
As $M \in \SL_+(\Gamma_n(\R))$, $\vert \gamma \vert^2 + \vert \alpha \vert^2=1$. Moreover $\alpha\overline{\gamma} = \alpha\gamma^{-1}\vert \gamma \vert^2$, which is a vector by Lemma~\ref{simulVn} and by the definition of $\SL_+(\Gamma_n(\R))$. Hence $\alpha\overline{\gamma}= (\alpha\overline{\gamma})^* = \gamma'\alpha^*$. This proves (iii). So suppose now $M(i_n)=i_n$. Then by Proposition~\ref{propels2}, $\Psi(M)$ fixes $0 \in \BQ^n$ and thus, by Proposition~\ref{defpsigamma}, $\beta+\gamma'+(\alpha-\delta')i_n=0$. So $\alpha=\delta'$ and $\beta=-\gamma'$ and $1= \alpha\delta^* - \beta\gamma^* =\vert \alpha \vert^2 + \vert \beta \vert^2 = 1$. Item (ii) now easily follows. Finally suppose $\Vert M \Vert^2 = 2$. It is easy to compute that $\vert \alpha - \delta' \vert^2 + \vert \beta + \gamma' \vert^2 = \Vert M \Vert^2 -2$. From this (i) follows immediately.  
\end{proof}

We have now all the ingredients to prove the main theorem of this section, which gives concrete formulas for the bisectors $\Sigma_{M}(i_n)$, defined in Section~\ref{prelimpoin}, in the upper half-space $\Hy^{n+1}$. The main idea is that the formulas for the bisectors in the ball model are easy to compute as in that model, bisectors equal isometric spheres, by Lemma~\ref{defpsigamma}. Unfortunately in the upper half-space the isometric sphere of an element $M \in \SL_+(\Gamma_n(\R))$ does not necessarily equal a bisector. However as the bisector is a concept purely based on the hyperbolic metric, the map $\eta_0^{-1}$ from Proposition~\ref{propels2} sends $\Sigma_{\Psi(M)}(0)$ to $\Sigma_{M}(i_n)$, as $\eta_0^{-1}(0)=i_n$. As $\Sigma_{M}(0)$ is also the isometric sphere, by Lemma~\ref{defpsigamma}, the formula for the centre and the radius of $\Sigma_{M}(0)$ may be easily computed and this gives us formulas for the bisector $\Sigma_{M}(i_n)$ in the next Theorem. Note that this theorem is exactly the same as \cite[Theorem 3.2]{algebrapaper}.

\begin{theorem}\label{maintheorem}
Let $M =
\begin{pmatrix} \alpha & \beta \\ \gamma & \delta \end{pmatrix} \in \SL_+(\Gamma_n(\R))$ and $M \not \in \SU_+(\Gamma_n(\R))$.
\begin{enumerate}
\item
$\Sigma_{M^{-1}}(i_n) $ is a Euclidean hemisphere if and
only if $|\alpha|^2+|\gamma|^2 \neq 1$. In this case, its
centre  and its radius are respectively given by
$P_{M^{-1}}=\frac{-(\beta^*\alpha'+\delta^*\gamma')}{|\alpha|^2+|\gamma|^2-1}$ and
$R^2_{M^{-1}}=\frac{1+\vert P_{M}\vert^2}{|\alpha|^2+|\gamma|^2}$.
\item
$\Sigma_{M^{-1}}(i_n)$ is a hyperplane if and only if
$|\alpha|^2+|\gamma|^2 = 1$. In this case
$Re(\overline{v}z)+\frac{|v|^2}{2}=0$, $z\in \C$ is
a  defining equation of $\Sigma_{M^{-1}}(i_n)$,  where
$v=\beta^*\alpha'+\delta^*\gamma'$.
\end{enumerate}
\end{theorem}

\begin{proof}
The proof is exactly the same as in \cite{algebrapaper}. Note that by \cite[Theorem 3.7]{ElsGrunMen}, if $M^{-1} \in \SL_+(\Gamma_n(\R))$, then $\beta^*\overline{\alpha^*}=\beta^*\alpha'$ and $\delta^*\overline{\gamma^*}=\delta^*\gamma'$ are vectors and hence $P_{M^{-1}}$ and $v$ are well defined. 
\end{proof}

Note that we exclude the case $M \in \SU_+(\Gamma_n(\R))$ in Theorem~\ref{maintheorem}, to ensure that the bisector of $i_n$ and $M^{-1}(i_n)$ exists.

\section{The algorithm}

In this section we describe an algorithm that computes a polyhedron of finite volume, containing a fundamental domain of a finite index subgroup $G$ of $\PSL_+(\Gamma_n(\Z))$. The basic idea of the algorithm is the same as in \cite{algebrapaper}. In the algorithm we have to order the elements in some way. As $G$ is a discrete group, every proper function will do the trick. The following lemma shows why we choose the norm function.

\begin{lemma}\label{rhogamma}
Let $M= \begin{pmatrix} \alpha & \beta \\ \gamma & \delta \end{pmatrix} \in \SL_+(\Gamma_n(\R)) \setminus \SU_+(\Gamma_n(\R))$. Then the following hold.
\begin{enumerate}[(i)]
\item $\vert P_{\Psi(M)} \vert^2 = \frac{2+\Vert M \Vert^2}{-2+\Vert M \Vert^2}$.
\item $R_{\Psi(M)}^2 = \frac{4}{-2 + \Vert M \Vert^2}$.
\item The function $\rho_{M} = 1 + R_{\Psi(M)} - \Vert P_{\Psi(M)} \Vert$ is a strictly decreasing function in $\Vert M \Vert$ on $\SL_+(\Gamma_n(\R)) \setminus  \SU_+(\Gamma_n(\R))$.
\end{enumerate}
\end{lemma} 

\begin{proof}
By Lemma~\ref{defpsigamma} (iii), $P_{\Psi(M)} = C^{-1}A$ and $R_{\Psi(M)} = \frac{1}{\vert C \vert}$. Using the explicit formulas for $A$ and $C$ from Lemma~\ref{defpsigamma}, item (i) and (ii) may be easily computed. 
Using this, we obtain $\rho_{M} = 1 - (\frac{2+\Vert M \Vert^2}{-2+ \Vert M \Vert^2})^{\frac{1}{2}}+2(-2+\Vert M \Vert^2)^{-\frac{1}{2}}$. In the proof of Lemma~\ref{su}, we have seen that $\Vert M \Vert^2 -2 = \vert \alpha-\delta' \vert^2 + \vert \beta + \gamma' \vert^2$. Thus $\Vert M \Vert^2 \geq 2$ for $M \in \SL_+(\Gamma_n(\R))$ and equality holds if and only if $M \in \SU_+(\Gamma_n(\R))$. So $\rho_{M}$ is a function in $M$ from $\left] 2, + \infty \right]$ to $\R$ such that $\rho_{M}' = -2 \Vert M \Vert ( \Vert M \Vert^2 - 2 ) ^{-\frac{3}{2}}(\Vert M \Vert^2 + 2 )^{-\frac{1}{2}}(-2+ \sqrt{\Vert M \Vert^2 +2})$. The latter shows that $\rho_{M}$ is strictly decreasing. 
\end{proof}

If we denote by $r$ the ray in $\BQ^{n+1}$ starting at $0$ and through the centre $P_{\Psi(M)}$ of the isometric sphere $\Sigma_{\Psi(M)}(0)$, then $\rho_{M}$ represents the distance between the intersection of $r$ with $\partial \BQ^{n+1}=\SQ^{n+1}$ and $\Sigma_{\Psi(M)}(0)$. So the intuitive idea of the algorithm is the following: the smaller $\Vert M \Vert^2$, i.e. the bigger $\rho_{M}$, the more the bisector $\Sigma_{\Psi(M)}(0)$, and hence also $\Sigma_{M}(i_n)$ contributes to the fundamental Dirichlet domain.

We now come to the main algorithm. Let $G$ be a congruence subgroup of level $l$ in $\PSL_+(\Gamma_n(\Z))$ and let $\F_{i_n}$ be a polyhedron in $\Hy^{n+1}$ containing a fundamental domain of the stabilizer $G_{i_n}$ of the point $i_n \in \Hy^{n+1}$. Thus $\F_{i_n}$ is given by a system of linear and quadratic inequalities with coefficients in $\overline{\Q}$.
Except for finitely many levels, the congruence subgroup $G$ will be torsion-free in which case one has $\F_{i_n}=\Hy^{n+1}$. For special cases one can do better than this; for instance, in the example presented in \Cref{sectionexample}, $G_{i_n}$ will be a finite reflection group and as our input $\F_{i_n}$, we will use its fundamental chamber.

As seen in Theorem~\ref{maintheorem}, the bisector $\Sigma_{M^{-1}}(i_n)$ of some element $M \in G$ can take the form of a hyperplane or a hemisphere. We first write a small algorithm searching for the hyperplanes appearing in the definition of the Dirichlet domain. Recall from (\ref{defbis}) that if $\Sigma_{M^{-1}}(i_n)$ is a hyperplane, then $D_{M^{-1}}(i_n)$ is the open half-space in $\Hy^{n+1}$ containing $i_n$ and bounded by $\Sigma_{M^{-1}}(i_n)$. We denote by $D_{M^{-1}}^{\R}(i_n)$ the open half-space in $\R^{n+1}$ containing the point $(0, \ldots, 0, 1)$ and bounded by the hyperplane $\Sigma_{M^{-1}}(i_n)$. Moreover we denote by $H$ the hyperplane in $\R^n$ defined by the equation $x_{n+1}=0$.

\begin{algorithm}\label{algorithm1}
The input of the algorithm is a congruence subgroup $G$ of level $l$ in $\PSL_+(\Gamma_n(\Z))$. Set $k=1$ and $B = \R^{n+1}$.
\begin{enumerate}[Step 1]
\item Compute $B= B \cap (\cap D_{M^{-1}}^{\R}(i_n))$, where $M$ runs over elements in $G$ and $M \not \in \PSU_+(\Gamma_n(\Z))$ such that $\Vert M \Vert^2 = k$ and $\vert \alpha(M) \vert^2 + \vert \gamma(M) \vert^2 = 1$.
\item If $\overline{B} \cap H$ is compact in $H$, stop and return $B$ and $K=k$.
\item If not, set $k=k+1$ and go to step 1.
\end{enumerate}
\end{algorithm}

\begin{proposition}
Algotithm~\ref{algorithm1} stops after finitely many steps.
\end{proposition}

\begin{proof}
Consider the cusp at $\infty$ of $G$ and let $H_{\infty}$ be a horosphere based at $\infty$. Note that the cusp at $\infty$ exists because $G$ is a finite index subgroup of $\PSL_+(\Gamma_n(\Z))$. Consider the subgroup $G_{\infty}$, i.e. the stabilizer of $\infty$. Elements $M$ of $G_{\infty}$ are such that $\gamma(M)=0$ and $\vert \alpha(M) \vert =1$. These are among the elements the algorithm runs through.
By \cite[Theorem 4.7.4]{ratcliffe}, $G_{\infty}$ acts by Euclidean isometries on $H_{\infty}$ and by \cite[Theorem 5.4.6]{ratcliffe} and \cite[Theorem 7.5.2]{ratcliffe} $G_{\infty}$ is a crystallographic group. The intersection of the half-spaces $D_M(i_n)$ for $M \in G_{\infty}$ with $H_{\infty}$ defines a Dirichlet domain (in the Euclidean sense) for the action of $G_{\infty}$ on $H_{\infty}$. This Dirichlet domain is compact by \cite[Theorem 7.5.1]{ratcliffe}, and thus has finitely many sides (see \cite[Theorem 6.3.6]{ratcliffe}).
Hence the algorithm has to stop after finitely many steps. 
\end{proof}

\begin{remark} To implement \Cref{algorithm1}, one needs a 
procedure to check if the convex polyhedron $\overline{B} \cap 
H$ is compact. As the algorithm deals with polyhedra in Euclidean space, this is possible. By \Cref{maintheorem}, the boundaries of the half-spaces $D_{M^{-1}}(i_n)$ appearing in the algorithm are hyperplanes. Hence the polyhedron $\overline{B} \cap 
H$ in $\R^n$ is described in every step by a finite system of linear inequalities $Ax \geq c$, for $A$ some real matrix, $c$ some real vector (both $A$ and $c$ have entries in $\overline{\Q}$) and $x$ the vector of unknowns. Consider the homogeneous cone $C$ over $\overline{B} \cap 
H$, i.e. place $\overline{B} \cap 
H$ in the hyperplane $x_{n+1}=1$ and take the cone from the origin over $\overline{B} \cap 
H$. Then $\overline{B} \cap 
H$ is compact if and only if the cone $C$ is sharp, i.e. contains no linear subspaces. The cone $C$ is given by a system of homogeneous linear inequalities $By \geq 0$. Now the cone $C$ is sharp if and only if the linear system $By=0$ has no non-zero solutions. The latter is a standard linear algebra problem which can be solved, for instance, via the Gauss method. 
\end{remark}

We now state the main algorithm computing a polyhedron containing the Dirichlet fundamental domain of a subgroup of finite index in $\PSL_+(\Gamma_n(\Z))$. Denote by $K$ the integer $k$ at which Algorithm~\ref{algorithm1} stops and let  $\F_{i_n}$ be defined as above. Recall, that if $\Sigma_{M^{-1}}(i_n)$ is a hemisphere, then $D_{M^{-1}}(i_n)$ is the open half-space in $\Hy^{n+1}$ containing $i_n$ and bounded by $\Sigma_{M^{-1}}(i_n)$. Similarly to Algorithm~\ref{algorithm1}, we denote by $D^{\R}_{M^{-1}}(i_n)$ the open subset of $\R^{n+1}$ which is the complement of the ball bounded by $\Sigma_{M^{-1}}(i_n)$. The hyperplane $H$ is defined as in the previous algorithm. Moreover we denote by $\F_{i_n}^{\R}$ the interior of the intersection of the boundary of $\F_{i_n}$ with $H$.

\begin{algorithm}\label{algorithm2}
The input of the algorithm is a congruence subgroup $G$ of level $l$ in $\PSL_+(\Gamma_n(\Z))$, the set $B$ obtained in Algorithm~\ref{algorithm1}, the set $\F_{i_n}$ and the integer $K$. Set $k=1$, $S= \lbrace M \mid \Vert M \Vert^2 \leq K$, $\vert \alpha(M) \vert^2 + \vert \gamma(M) \vert^2 = 1 \rbrace$ and $BF=H \cap B \cap {\F}^{\R}_{i_n}$.
\begin{enumerate}[Step 1]
\item For every $M \not \in \PSU_+(\Gamma_n(\Z))$, such that $\Vert M \Vert^2 = k$ and $\vert \alpha(M) \vert ^2 + \vert \gamma(M) \vert ^2 \neq 1$, if $BF \cap {D}^{\R}_{M^{-1}}(i_n)  \neq BF$, then $S = S \cup \lbrace M \rbrace$ and $BF= BF \cap {D}^{\R}_{M^{-1}}(i_n)$.
\item If $BF = \emptyset$, stop and return $S$.
\item If not, set $k=k+1$ and go to step $1$.
\end{enumerate}
\end{algorithm}

\begin{remark}
The definition of $\SL_+(\Gamma_n(\Z))$ is not suitable for algorithms as $\Gamma_n(\Z)$ is defined as the set of all possible products of invertible vectors such that the product has integral coefficients. Therefore to realize Algorithms~\ref{algorithm1} and \ref{algorithm2}, it is better to work with the alternative definition of $\SL_+(\Gamma_n(\Z))$ given in \cite[Definition 3.1]{ElsGrunMen} and mentioned before.  
\end{remark}

The set $BF$ is a real semi-algebraic set. Indeed, $H$ is a hyperplane given by $x_{n+1}=0$ and the sets $B$ and $\F^{\R}_{i_n}$ are given by finitely many linear and quadratic inequalities. The set ${D}^{\R}_{M^{-1}}(i_n)$ is also given by a quadratic inequality and thus $BF$ is a real semi-algebraic set. Moreover the linear and quadratic inequalities describing $BF$ have integer coefficients. Hence there are algorithms available to check if $BF$ is empty. For more information on this, we refer to \cite[Chapter 13]{BaRiRo}, and in particular \cite[Algorithm 13.1]{BaRiRo}.

\begin{proposition}\label{fundgen}
The set $S$ is finite and the group $\langle M \cup G_{i_n} \mid M \in S \rangle$ is a subgroup of finite index in $G$. Moreover the closure of $\tilde{\F} = \bigcap_{M \in S} D_{M^{-1}}(i_n) \cap \F_{i_n}$ is a finite-sided polyhedron containing the Dirichlet fundamental domain $\F$ of $G$.
\end{proposition}

\begin{proof}
By Theorem~\ref{finitelygen} and Corollary~\ref{geomfinite}, the fundamental domain of $\PSL_+(\Gamma_n(\Z))$ has finite volume and finitely many sides and hence also $G$ has finite covolume and its fundamental domain has finitely many sides. Therefore Algorithm~\ref{algorithm2} has to stop after finitely many steps. As $G$ is a discrete subgroup of $\PSL_+(\Gamma_n(\R))$, for every $k \geq 1$, there are only finitely many $M \in G$ such that $\Vert M \Vert^2 = k$. Thus the set $S$ is finite. 

By construction $\tilde{\F}$ is clearly a finite-sided convex polyhedron, which contains the Dirichlet fundamental polyhedron $\F$ with centre $i_n$ for $G$. Moreover $\tilde{\F}$ contains the fundamental Dirichlet domain $\F'$ of $\langle M \cup G_{i_n} \mid M \in S \rangle$. 
Since \Cref{algorithm2} concludes with $BF = \emptyset$, the intersection of the closure of $\tilde{\F}$ with
$\partial \Hy^{n+1}$ is finite and, hence, the finite-sided polyhedron $\tilde{\F}$ has finite volume. Thus we have
$$\F \subseteq \F' \subseteq \tilde{\F},$$
and $\tilde{\F}$ has finite volume. 
Thus,  $\F'$ has finite volume and by \cite[Theorem 6.7.3]{ratcliffe}, $\langle M \cup G_{i_n} \mid M \in S \rangle$ has finite index in $G$. 
\end{proof}

As we noted above, the polyhedron $\tilde{\F}$ has finite volume and, hence, only finitely
many ideal vertices. While the number of ideal boundary points of an actual fundamental
polyhedron of G might be smaller, this is not a problem.

\begin{remark}\label{algosimple}
Note that in some cases, e.g. if $G=\PSL_+(\Gamma_n(\Z))$, it is possible to simplify Algorithm~\ref{algorithm1} and 
\ref{algorithm2} as follows. If $M\in 
\PSL_+(\Gamma_n(\Z))$, then $\vert a(M) \vert^2 \in \Z$. The same is 
true for $\vert c(M) \vert^2$. Thus $\vert a(M) \vert^2 + \vert c(M) 
\vert^2 = 1$ if and only $a(M)=0$ or $c(M)=0$ and the one different 
from $0$ is of norm $1$. Suppose $a(M)=0$ and let $M' = 
\begin{pmatrix} 0 & -1 \\ 1 & 0 \end{pmatrix} M $. Then $c(M')=0$ 
and by Theorem~\ref{maintheorem} $\Sigma_{M^{-1}}(i_n) = 
\Sigma_{M'^{-1}}(i_n)$. Thus, without loss of generality, in 
Algorithm~\ref{algorithm1}, it suffices to run step 1 through the matrices 
$M \in G$ such that $\vert a(m) \vert^2=1$ and $c(M)=0$. In 
Algorithm~\ref{algorithm2}, $S$ may be adapted accordingly. 
\end{remark}

\section{Pulling back subgroups of \texorpdfstring{$\PSL_+(\Gamma_4(\R))$ to $\PSL_2(\h(\R))$}{LG}}

As shown in the previous section, we can find generating sets for subgroups of finite index of $\PSL_+(\Gamma_4(\Z))$. However, as explained in the introduction, we are interested in finding generators
of finite index subgroups of discrete subgroups of the groups $\PSL_2((\frac{x,y}{\Q}))$ for $x,y \in \Z_{<0}$. In particular, by \cite[Theorem 2.5]{eiskiefvgel} we are interested in the cases $x=y=-1$, $x=-1$, $y=-3$ and $x=-2$, $y=-5$. 

As both $\PSL_+(\Gamma_4(\R))$ and $\PSL_2(\h(\R))$ are isomorphic to $\Iso^+(\Hy^5)$, they are isomorphic. We will now describe this isomorphism.
Details on this can be found in a more general context in \cite[Section 6]{ElsGrunMen}. First we define the following two elements
\begin{eqnarray*}
\varepsilon_1 & = & \frac{1}{2}(1+i_1i_2i_3),\\
\varepsilon_2 & = & \frac{1}{2}(1-i_1i_2i_3).
\end{eqnarray*}
Note that $\varepsilon_1$ and $\varepsilon_2$ are elements of the centre of $\Cliff_4(\R)$. Moreover we have the following:
$$ \varepsilon_i^2= \varepsilon_i,\ \overline{\varepsilon_i}=\varepsilon_i \textrm{ for }i=1,2,\ \varepsilon_1^*=\varepsilon_2,\ \epsilon_1\epsilon_2=0,\ \varepsilon_1+\varepsilon_2=1.$$
Thus $\varepsilon_1$ and $\varepsilon_2$ are orthogonal idempotents in $\Cliff_4(\R)$. Now every element $\alpha \in \Cliff_4(\R)$ may be expressed as
$$\alpha=\varepsilon_1a_1 + \varepsilon_2a_2,$$
where $a_1, a_2 \in \Cliff_3(\R)$. Indeed let $\alpha=\alpha_0+\alpha_1i_1+\alpha_2i_2+\alpha_3i_3+\alpha_4i_1i_2
+\alpha_5i_1i_3+\alpha_6i_2i_3+\alpha_7i_1i_2i_3$ and $a_1=x_0+x_1i_1+x_2i_2+x_3i_1i_2$ and $a_2=y_0+y_1i_1+y_2i_2+y_3i_1i_2$. Then $\alpha=a_1\varepsilon_1 + a_2\varepsilon_2$ if and only if
\begin{equation}\label{epsiloneq}
\begin{array}{lcl}
\alpha_0 = \frac{x_0+y_0}{2},  & \quad & \alpha_4 = \frac{x_3+y_3}{2},\\
\alpha_1 = \frac{x_1+y_1}{2},  & \quad & \alpha_5 = \frac{x_2-y_2}{2},\\
\alpha_2 = \frac{x_2+y_2}{2},  & \quad & \alpha_6 = \frac{-x_1+y_1}{2},\\
\alpha_3 = \frac{-x_3+y_3}{2},  & \quad & \alpha_7 = \frac{x_0-y_0}{2}.
\end{array}
\end{equation}
It is well known (see for example \cite{1985ahlfors}) and easy to prove that $\Cliff_3(\R)$ is isomorphic to $\h(\R)$. Denote this isomorphism by $\omega$. It is clear that $\omega(1)=1$, $\omega(i_1)=i$, $\omega(i_2)=j$ and $\omega(i_1i_2)=k$. Thus we can define an algebra homomorphism
\begin{equation}
\chi_1: \Cliff_3(\R) \rightarrow \h(\R), \ \alpha \mapsto \omega(a_1).
\end{equation}
Of course one can extend $\chi$ to the corresponding matrix algebras
\begin{equation}
\tilde{\chi}: M_2(\Cliff_4(\R)) \rightarrow M_2(\h(\R)).
\end{equation}

We now restrict $\tilde{\chi}$ to the group $\SL_+(\Gamma_4(\R))$. The proof of the following lemma can be found in \cite[Section 6]{ElsGrunMen}. 

\begin{lemma}\label{isomsl+sl}
The homomorphism $\tilde{\chi}$ restricts to an isomorphism 
\begin{equation}\label{defchi}
\chi: \SL_+(\Gamma_4(\R)) \rightarrow \SL_2(\h(\R)).
\end{equation}
\end{lemma}

As we are interested in orders in $\SL_2(\h(\R))$ or more precisely subgroups of finite index of $\SL_2(\h(\Z))$, we consider the restriction of $\chi$ to a group commensurable with $\SL_+(\Gamma_4(\Z))$. Therefore define
$$\tilde{\Gamma}_4(\Z)= \lbrace \epsilon_1 a + \epsilon_2 b \mid a,b, \in \Cliff_3(\Z), \epsilon_1 a + \epsilon_2 b \in \Gamma_4(\R)  \rbrace$$
and
\begin{eqnarray*}
\SL_+(\tilde{\Gamma}_4(\Z)) & = & \left \lbrace \begin{pmatrix} \alpha & \beta \\ \gamma & \delta \end{pmatrix} \mid \alpha,\beta,\gamma,\delta \in \tilde{\Gamma}_4(\Z); \alpha\delta^*-\beta\gamma^*=1; \alpha\beta^*, \gamma\delta^*, \gamma^*\alpha, \delta^*\beta \in \V^4(\Z) \right \rbrace.
\end{eqnarray*}

It is clear that $\SL_+(\tilde{\Gamma}_4(\Z))$ contains $\SL_+(\Gamma_4(\Z))$ as a subgroup. For the next lemma we need to compute the inverse of the isomorphism $\chi$. By \cite{ElsGrunMen}, we know the pre-images of the generators of $\SL_2(\h(\R))$. First it is easy to see that
$$\chi\left(\begin{pmatrix} 0 & -1 \\ 1 & 0 \end{pmatrix} \right) = \begin{pmatrix} 0 & -1 \\ 1 & 0 \end{pmatrix}.$$
Then, for every $b \in \h(\R)$,
$$\chi\left(\begin{pmatrix} 1 & 0 \\ \beta & 1 \end{pmatrix}\right)= \begin{pmatrix} 1 & 0 \\ b & 1 \end{pmatrix},\ \chi\left(\begin{pmatrix} 1 & \beta \\ 0 & 1 \end{pmatrix}\right)=  \begin{pmatrix} 1 & b \\ 0 & 1 \end{pmatrix},$$
where $\beta \in \V^4(\R)$ unique such that $\chi_1(\beta)=b$. Finally, for $d \in \h(\R)$ with $d\overline{d}=1$,
$$\chi\left(\begin{pmatrix} \varepsilon_1+ \varepsilon_2\overline{\delta} & 0 \\ 0 & \varepsilon_1\delta + \varepsilon_2  \end{pmatrix}\right)= \begin{pmatrix} 1 & 0 \\ 0 & d \end{pmatrix},$$
where $\delta \in \V^4(\R)$ unique such that $\chi_1(\delta)=d$. By the alternative definition  of $\SL_+(\Gamma_n(\R))$ \cite[Definition 3.1]{ElsGrunMen} (see also the remark under Definition~\ref{defcliffmoeb}), the above matrix is clearly contained in $\SL_+(\Gamma_4(\R))$. Note that $\delta\overline{\delta}=1$. Indeed, write $\delta=\varepsilon_1d_1+\varepsilon_2d_2$. Then $\chi_1(\delta)=d$ if and only if $d_1=\omega^ {-1}(d)$ and hence $d_2=\omega^{-1}(d)^*$. Thus $$\delta\overline{\delta}=(\varepsilon_1d_1+\varepsilon_2d_2)(\varepsilon_1\overline{d_1}+\varepsilon_2\overline{d_2})
= \varepsilon_1 d_1\overline{d_1} + \varepsilon_2 d_2\overline{d_2}
= \varepsilon_1 \omega^{-1}(d\overline{d}) + \varepsilon_2\omega^{-1}(d\overline{d})^* = \varepsilon_1 + \varepsilon_2 = 1.$$
Take an arbitrary matrix $\begin{pmatrix} a & b \\ c & d \end{pmatrix} \in \SL_2(\h(\R))$. Let $\sigma$ be as defined in (\ref{defsigma}). We want to decompose the matrix into a product of generators. Therefore we essentially need to distinguish three cases. 

\noindent \underline{Case 1:} $c=0=b$. Then $a$ and $d$ are different from $0$ and we have
$$\begin{pmatrix} a & 0 \\ 0 & d \end{pmatrix} 
= \begin{pmatrix} 1 & -a \\ 0 & 1 \end{pmatrix}
\begin{pmatrix} 1 & 0 \\ a^{-1}-1 & 1 \end{pmatrix}
\begin{pmatrix} 1 & 1 \\ 0 & 1 \end{pmatrix}
\begin{pmatrix} 1 & 0 \\ a-1 & 1 \end{pmatrix}
\begin{pmatrix} 1 & 0 \\ 0 & ad \end{pmatrix}.$$
Let
\begin{equation}
\begin{array}{lcl}
\alpha \in \V^4(\R) & \textrm{ unique such that } & \chi_1(\alpha)=a\\
\tau \in \V^4(\R) & \textrm{ unique such that } & \chi_1(\tau)=ad.
\end{array}
\end{equation}
It is easy to see that $\alpha^{-1} \in \V^4(\R)$ and $\chi_1(\alpha^{-1})=a^{-1}$ and $\alpha -1 \in \V^4(\R)$ and $\chi_1(\alpha-1))=a-1$. By the above we may compute the pre-images of the matrices appearing in the decomposition and we get
\begin{equation}\label{preimagechi1}
\chi^{-1}\left(\begin{pmatrix} a & 0 \\ 0 & d \end{pmatrix}\right)
= \begin{pmatrix} \varepsilon_1\alpha+\varepsilon_2\alpha\overline{\tau} & 0 \\ 
0 & 
\varepsilon_1\alpha^{-1}\tau + \varepsilon_2\alpha^{-1} 
\end{pmatrix}
\end{equation}

\noindent \underline{Case 2:} $c=0$ and $b$, $a$ and $d$ are different from $0$. We have
$$\begin{pmatrix} a & b \\ 0 & d \end{pmatrix} 
= \begin{pmatrix} 1 & -a \\ 0 & 1 \end{pmatrix}
\begin{pmatrix} 1 & 0 \\ a^{-1}-1 & 1 \end{pmatrix}
\begin{pmatrix} 1 & 1 \\ 0 & 1 \end{pmatrix}
\begin{pmatrix} 1 & 0 \\ a-1 & 1 \end{pmatrix}
\begin{pmatrix} 1 & 0 \\ 0 & ad \end{pmatrix}
\begin{pmatrix} 1 & a^{-1}b \\ 0 & 1 \end{pmatrix}.$$ Define again
\begin{equation}
\begin{array}{lcl}
\alpha \in \V^4(\R) & \textrm{ unique such that } & \chi_1(\alpha)=a\\
\tau \in \V^4(\R) & \textrm{ unique such that } & \chi_1(\tau)=ad\\
\eta \in \V^4(\R) & \textrm{ unique such that } & \chi_1(\eta)=a^{-1}b.
\end{array}
\end{equation}
By the same computations as above, we find
\begin{equation}\label{preimagechi2}
\chi^{-1}\left(\begin{pmatrix} a & b \\ 0 & d \end{pmatrix}\right)
= \begin{pmatrix} \varepsilon_1\alpha+\varepsilon_2\alpha\overline{\tau} & \epsilon_1\alpha\eta+
\varepsilon_2\alpha\eta\overline{\tau} \\ 
0 & 
\varepsilon_1\alpha^{-1}\tau + \varepsilon_2\alpha^{-1}
\end{pmatrix}
\end{equation}

\noindent \underline{Case 3:} $c \neq 0$. We finally have 
$$\begin{pmatrix} a & b \\ c & d \end{pmatrix} 
= \begin{pmatrix} 1 & ac^{-1} \\ 0 & 1 \end{pmatrix}
\begin{pmatrix} 0 & -1 \\ 1 & 0 \end{pmatrix}
\begin{pmatrix} 1 & -c \\ 0 & 1 \end{pmatrix}
\begin{pmatrix} 1 & 0 \\ c^{-1}-1 & 1 \end{pmatrix}
\begin{pmatrix} 1 & 1 \\ 0 & 1 \end{pmatrix}
\begin{pmatrix} 1 & 0 \\ c-1 & 1 \end{pmatrix}
\begin{pmatrix} 1 & 0 \\ 0 & \sigma \end{pmatrix}
\begin{pmatrix} 1 & c^{-1}d \\ 0 & 1 \end{pmatrix}.$$
As we are working in $\SL_2(\h(\R))$, $\vert \sigma \vert^2 =1$. Let
\begin{equation}
\begin{array}{lcl}
\tau \in \V^4(\R) & \textrm{ unique such that } & \chi_1(\tau)=ac^{-1}\\
\eta \in \V^4(\R) & \textrm{ unique such that } & \chi_1(\eta)=c^{-1}d\\
\mu \in \V^4(\R) & \textrm{ unique such that } & \chi_1(\mu)=\sigma\\
\gamma \in \V^4(\R) & \textrm{ unique such that } & \chi_1(\gamma)=c.
\end{array}
\end{equation}
By the above, we can compute the pre-images of the seven matrices appearing in the decomposition of $\begin{pmatrix} a & b \\ c & d \end{pmatrix}$. Taking the product again, we obtain that
\begin{equation}\label{preimagechi3}
\chi^{-1}\left(\begin{pmatrix} a & b \\ c & d \end{pmatrix}\right)
= \begin{pmatrix} \varepsilon_1\tau\gamma+\varepsilon_2\tau\gamma\overline{\mu} & \epsilon_1(\tau\gamma\eta-\gamma^{-1}\mu)+
\varepsilon_2(\tau\gamma\eta\overline{\mu}-\gamma^{-1}) \\ 
\varepsilon_1\gamma +\varepsilon_2\gamma\overline{\mu} & 
\varepsilon_1\gamma\eta + \varepsilon_2\gamma\eta\overline{\mu}
\end{pmatrix}
\end{equation}

We have now all the ingredients to prove the following lemma.

\begin{lemma}\label{isooverZ}
The isomorphism $\chi$ restricts to an isomorphism from $\SL_+(\tilde{\Gamma}_4(\Z))$ to $\SL_2(\h(\Z))$. 
\end{lemma}

\begin{proof}
It is clear that $\chi(\SL_+(\tilde{\Gamma}_4(\Z))) \subseteq \SL_2(\h(\Z))$. To prove the converse, consider an element $\begin{pmatrix} a & b \\ c & d \end{pmatrix}$ in $\SL_2(\h(\R))$ and consider its pre-images given by (\ref{preimagechi1}), (\ref{preimagechi2}) or (\ref{preimagechi3}). These are products of pre-images of the generators of $\SL_2(\h(\R))$ and as the latter are contained in $\SL_+(\Gamma_4(\R))$, the pre-images given by (\ref{preimagechi1}), (\ref{preimagechi2}) or (\ref{preimagechi3}) are contained in $\SL_+(\Gamma_4(\R))$. Thus the different entries of the pre-images given by (\ref{preimagechi1}), (\ref{preimagechi2}) or (\ref{preimagechi3}) are elements of $\Gamma_4(\R)$. It remains to prove that the different entries may be expressed as  $\varepsilon_1a_1 +\varepsilon_2a_2$ with $a_1, a_2 \in \Cliff_3(\Z)$. Therefore note, that if $\alpha \in \Cliff_4(\R)$, then the expression of $\alpha$ as $\varepsilon_1a_1+\varepsilon_2a_2$ with $a_1,a_2 \in \Cliff_3(\R)$ is unique. Moreover $a_1$ is given by $\omega^{-1}\chi_1(\alpha)$ and $a_2$ by $(\omega^{-1}\chi_1(\alpha^{*}))^{*}$. As $\omega$ clearly restricts to an isomorphism from $\Cliff_3(\Z)$ to $\h(\Z)$, to prove that $\alpha$ may be expressed as $\varepsilon_1a_1+\varepsilon_2a_2$ with $a_1,a_2 \in \Cliff_3(\Z)$, it is enough to prove that $\chi_1(\alpha)$ and $\chi_1(\alpha^{*})$ are elements of $\h(\Z)$.  So we first consider the case when $b=c=0$ and hence the pre-images are given by (\ref{preimagechi1}). Then $\chi_1(\varepsilon_1\alpha+\varepsilon_2\alpha\overline{\tau})=a \in \h(\Z)$. As $\alpha, \tau \in \V^4(\Z)$, $(\alpha\overline{\tau})^{*}=\overline{\tau}\alpha$ and thus $\chi_1((\alpha\overline{\tau})^{*})=\overline{d}\vert a \vert^2$, which is an element of $\h(\Z)$. Concerning the other entry, $\chi_1(\alpha^{-1}\tau)=d \in \h(\Z)$ and $\chi_1(\alpha^{-1})=a^{-1}$, which is also an element of $\h(\Z)$ as the original matrix is contained in $\SL_2(\h(\Z))$.  The case $c=0 \neq b$ is dealt with in exactly the same way. The general case described in (\ref{preimagechi3}) asks for a bit more reflection. First it is clear that $\chi_1(\tau\gamma), \chi_1(\gamma),\chi_1(\gamma\eta) \in \h(\Z)$. The upper right entry gives $\chi_1(\tau\gamma\eta-\gamma^{-1}\mu)=ac^{-1}d-c^{-1}\sigma$, which equals $b$ according to (\ref{defsigma}). 
We now compute the coefficients of $\varepsilon_2$. As $\vert \sigma \vert = \Delta =1$, $\overline{\sigma}=\sigma^{-1}$. Moreover the determinant of the pre-images given by (\ref{preimagechi3}) equals $1+\tau\gamma(\overline{\mu}\eta-\eta\overline{\mu})\gamma$, which implies that $\eta$ and $\overline{\mu}$ commute. Applying this we get 
$\chi_1((\tau\gamma\overline{\mu})^{*})=\sigma^{-1}cac^{-1}$, $\chi_1((\gamma\overline{\mu})^{*})=\sigma^{-1}c$, $\chi_1((\gamma\eta\overline{\mu})^{*})=c^{-1}d\sigma^{-1}c$ and $\chi_1((\tau\gamma\eta\overline{\mu}-\gamma^{-1})^{*})=
c^{-1}d\sigma^{-1}cac^{-1}-c^{-1}$. The latter expressions appear like this in the inverse of the matrix given in (\ref{inversdeltamatrix}). Thus by Lemma~\ref{unitinMHZ}, they are elements in $\h(\Z)$.
\end{proof}

By Lemma~\ref{isooverZ}, to get generators for a subgroup $G$ of finite index in $\SL_2(\h(\Z))$, it is sufficient to find generators for $\chi^{-1}(G)$ and apply $\chi$ to these generators. Recall that we described the algorithms~\ref{algorithm1} and \ref{algorithm2} for subgroups of $\SL_+(\Gamma_4(\Z))$ and $\chi^{-1}(G)$ is a subgroup of $\SL_+(\tilde{\Gamma}_4(\Z))$, which contains $\SL_+(\Gamma_4(\Z))$ as a subgroup. The next lemma shows that, if one works with finite index subgroups, this makes no difference. 

\begin{lemma}\label{finiteindexorders}
The group $\SL_+(\Gamma_4(\Z))$ is a subgroup of finite index of $\SL_+(\tilde{\Gamma}_4(\Z))$.
\end{lemma}

\begin{proof}
Recall from Section~\ref{Cliffordsection} that we denote by $\Cliff_4(\Z)$ the module of integral combinations of the elements $i_1, i_2, i_3$. Denote by $\tilde{\Cliff}_4(\Z)$ the set of elements of the form $\varepsilon_1 a + \varepsilon_2 b$ with $a,b \in \Cliff_3(\Z)$. Then it is easy to see that both $\Cliff_4(\Z)$ and $\tilde{\Cliff}_4(\Z)$ are orders in $\Cliff_4(\Q)$, where the latter was defined in Remark~\ref{rationalClifford}. In the same way the matrix sets $M_2(\Cliff_4(\Z))$ and $M_2(\tilde{\Cliff}_4(\Z))$ are orders in $M_2(\Cliff_4(\Q))$. Notice that $M_2(\Cliff_4(\Z))$ has additive index $2$ in $M_2(\tilde{\Cliff}_4(\Z))$. By a classical argument one then shows that the multiplicative index of $\SL_+(\Gamma_4(\Z))$ in $\SL_+(\tilde{\Gamma}_4(\Z))$ is also finite. 
\end{proof}

This gives a method for finding generators of finite index subgroups of $\SL_2(\h(\Z))$. Note that as in Remark~\ref{abuse}, we abuse the notation and denote by $\chi$ the map defined in (\ref{defchi}) as well as its restriction to the projective groups $\PSL_+(\Gamma_4(\R))$ and $\PSL_2(\h(\R))$.

\begin{method}\label{methodeprincipale}
Let $G$ be a subgroup of finite index in $\SL_2(\h(\Z))$. Construct the group $\tilde{G} = \chi^{-1}(G) \cap \SL_+(\Gamma_4(\Z))$, where $\chi$ is defined in (\ref{defchi}). Consider the projection $\Pi(\tilde{G})$, where $\Pi$ is defined in (\ref{projection}). Apply algorithms~\ref{algorithm1} and \ref{algorithm2} for $n=4$ to $\Pi(\tilde{G})$ to get a finite generating set $S$ of a finite index subgroup of $\Pi(\tilde{G})$. The set $\chi\Pi^{-1}(S)$ is a finite generating set of a finite index subgroup of $G$.
\end{method}

\begin{proof}
As $G$ is of finite index in $\SL_2(\h(\Z))$, $\chi^{-1}(G)$ is of finite index in $\SL_+(\tilde{\Gamma}_4(\Z))$ by Lemma~\ref{isooverZ}. Thus, by Lemma~\ref{finiteindexorders}, $\tilde{G}$ is of finite index in $\SL_+(\Gamma_4(\Z))$. Finally $\chi(\tilde{G})$ has finite index in $\SL_2(\h(\Z))$ and is contained in $G$. Hence $\chi(\tilde{G})$ has finite index in $G$. Note that passing to the projective case does not affect the finite index. 
\end{proof}

Method~\ref{methodeprincipale} thus shows that finding generators for $\tilde{G}$ and mapping them back into $\SL_2(\h(\Z))$ is sufficient for our purpose.

\section{Generators for subgroups of finite index in \texorpdfstring{$\SL_2\left(\left(\frac{x,y}{\Z}\right)\right)$}{LG} with \texorpdfstring{$x$}{LG} and \texorpdfstring{$y$}{LG} negative integers}\label{sectiongenquat}

In this section we consider the problem of finding generators for a subgroup of finite index in $\SL_2\left(\left(\frac{x,y}{\Z}\right)\right)$ for general negative integers $x,y$, that are not simultaneously $-1$. The main idea is to consider different quaternion algebras over algebraic extensions of $\Q$ which are isomorphic to the classical quaternion algebra with $a=b=-1$. For this recall that if $K$ is a field of characteristic different from $2$, then for any $x,y,z,t \in K^{*}$, $$\left(\frac{x,y}{K}\right) \cong \left(\frac{z^2x,t^2y}{K}\right).$$

In particular the latter implies that over an algebraically closed field, all quaternion algebras are isomorphic. The field $\R$ is not algebraically closed, but it is easy to see that over $\R$ all totally definite quaternion algebras are isomorphic. We consider now quaternion algebras over $\Q$. Let $\left(\frac{x,y}{\Q}\right)$ be a totally definite quaternion algebra. Hence $x$ and $y$ are negative. Of course this algebra is not isomorphic to the classical quaternion algebra $\h(\Q)$. However as algebras over $\Q(\sqrt{\vert x \vert}, \sqrt{\vert y \vert})$, both algebras are isomorphic. This is however too big for our purpose, as we are interested in $\left(\frac{x,y}{\Q}\right)$ and not $\left(\frac{x,y}{\Q(\sqrt{\vert x \vert}, \sqrt{\vert y \vert})}\right)$. Therefore we need the following lemma. 

\begin{lemma}\label{quatisoQ}
Let $x$ and $y$ be negative integers. Then  
$$\h_{x,y}(\Q) := \lbrace a_0 + a_1\sqrt{\vert x \vert}i+a_2\sqrt{\vert y \vert}j+a_3\sqrt{\vert xy\vert }k \mid a_0,a_1,a_2,a_3 \in \Q \textrm{ and } i^2=j^2=-1, ij=-ji=k \rbrace$$
is a subalgebra of $\h(\R)$, which is isomorphic to the totally definite quaternion algebra $\left(\frac{x,y}{\Q}\right)$. 
\end{lemma}

\begin{proof}
It is easy to prove that $\h_{x,y}(\Q)$ is a subalgebra of $\h(\R)$. The isomorphism between $\h_{x,y}(\Q)$ and $\left(\frac{x,y}{\Q}\right)$ is obtained by sending $1$ to $1$, $\sqrt{x}i$ to $i$ and $\sqrt{y}j$ to $j$.
\end{proof}

Let us denote the isomorphism described in the proof above by $\Lambda_1$. So
\begin{align*}
\Lambda_1:  \h_{x,y}(\Q) & \rightarrow \left(\frac{x,y}{\Q}\right) \\
a_0 + a_1\sqrt{\vert x \vert}i+a_2\sqrt{\vert y \vert}j+a_3\sqrt{\vert xy\vert }k & \mapsto  a_0 + a_1i+a_2j+a_3k
\end{align*}

Denote by $\tilde{\Lambda}$ the extensions of $\Lambda_1$ to the matrix algebras, i.e.
$$\tilde{\Lambda}: M_2(\h_{x,y}(\Q))  \rightarrow  M_2\left(\left(\frac{x,y}{\Q}\right)\right),$$
and let $\Lambda$ be the restriction of $\tilde{\Lambda}$ to the group $\SL_2(\h_{x,y}(\Q))$, i.e.
$$\Lambda: \SL_2(\h_{x,y}(\Q))  \rightarrow  \SL_2\left(\left(\frac{x,y}{\Q}\right)\right).$$

Consider the order $\left(\frac{x,y}{\Z}\right)$ in $\left(\frac{x,y}{\Q}\right)$. Then $\tilde{\Lambda}^{-1}\left(\left(\frac{x,y}{\Z}\right)\right)$ equals the order $\h_{x,y}(\Z)$ which is a discrete subgroup of $\h(\R)$ and $\Lambda^{-1}\left(\SL_2\left(\left(\frac{x,y}{\Z}\right)\right)\right)$ equals $\SL_2\left(\h_{x,y}(\Z)\right)$

Algorithms~\ref{algorithm1} and \ref{algorithm2} are stated for congruence subgroups of $\PSL_+(\Gamma_n(\Z))$.  Let $G$ be a congruence subgroup of a given level in $\SL_2\left( 
\left(\frac{x,y}{\Z}\right)\right)$. Then the group 
$\chi^{-1}\Lambda^{-1}(G)$, where $\chi$ was defined in (\ref{defchi}) is an arithmetic lattice of $\PSL_+(\Gamma_n(\R))$, but which is not contained in $\PSL_+(\Gamma_n(\Z))$. Thus this group cannot be an input group for Algorithm~\ref{algorithm1} and \ref{algorithm2}. Nevertheless both algorithms can be easily modified to use as inputs discrete subgroups of $\PSL_+(\Gamma_n(\R))$, hence $\Z$-orders in $\PSL_+(\Gamma_n(\R))$, that can be presented in some fixed algorithmic form.
In fact, as the input set is an arithmetic lattice, all the matrix coefficients are elements of an algebraic number field. By restriction of scalars, an arithmetic lattice can be seen as $\Z$-lattice of higher rank and hence a standard Turing machine may be used. For more information on algorithms with algebraic numbers, we refer the interested reader to \cite{Cohen}.

With this at hand we get the following method for computing a finite set of generators of a finite index subgroup of $\SL_2\left(\frac{x,y}{\Z}\right)$.

\begin{method}\label{methodesecondaire}
Let $x$ and $y$ be negative integers and let $\left(\frac{x,y}
{\Z}\right)$ be the totally definite quaternion ring over $\Z$. 
Let $G$ be a congruence subgroup of level $l$ of $\SL_2\left( 
\left(\frac{x,y}{\Z}\right)\right)$. Construct the group 
$\chi^{-1}\Lambda^{-1}(G)$, where $\chi$ was defined in (\ref{defchi}). Apply
algortithms~\ref{algorithm1} and \ref{algorithm2}, modified as explained above, to $\chi^{-1}\Lambda^{-1}
(G)$ to get a generating set $S$ of a finite index subgroup of
$\chi^{-1}\Lambda^{-1}(G)$. Then $\Lambda\chi(S)$ is a finite generating set 
of a finite index subgroup of $G$. 
\end{method}

\begin{proof}
The proof is similar to the proof of Method~\ref{methodeprincipale}.
\end{proof}

\section{A Short Example}\label{sectionexample}

In this section we will present an example: the order $\SL_2(\h(\Z))$ in the $2$-by-$2$ matrix ring over the classical rational quaternion algebra $\h(\Q)$. So let $G$ be the group $\SL_+(\Gamma_4(\Z))$. We first determine $G_{i_4}$, the stabilizer of $i_4$, in order to compute a fundamental domain $\F_{i_4}$. By Lemma~\ref{su}, $M \in G_{i_4}$ if and only if $M=\begin{pmatrix}  \alpha & \gamma \\ \gamma' &  \alpha' \end{pmatrix}$ with $\vert \alpha \vert^2+ \vert \gamma \vert^2=1$. As $\vert \alpha \vert, \vert \gamma \vert \in \Z$, the latter is true if and only if $\vert \alpha \vert = 1$ and $\vert \gamma \vert =0$ or the inverse. Observe that $\vert \alpha \vert = 1 $ if and only if $\alpha \in \lbrace \pm 1, \pm i_1, \pm i_2,\pm i_3, \pm i_1i_2, \pm i_1i_3, \pm i_2i_3, \pm i_1i_2i_3  \rbrace$. Hence 
$$G_{i_4} = \left\langle \begin{pmatrix} 0 & 1 \\ -1 & 0 \end{pmatrix}, \begin{pmatrix} i_1 & 0 \\ 0 & -i_1 \end{pmatrix}, \begin{pmatrix} i_2 & 0 \\ 0 & -i_2 \end{pmatrix} , \begin{pmatrix} i_3 & 0 \\ 0 & -i_3 \end{pmatrix}  \right\rangle.$$
Note that the first matrix gives a reflection in the unit hemisphere in $\Hy^5$, while the other matrices give reflections in planes containing $0$ and perpendicular to $\partial \Hy^5$. It is now easy to see that 
$$\F_{i_4} = \lbrace z_0+z_1i_1+z_2i_2+z_3i_3+z_4i_4 \mid z_0^2+z_1^2+z_2^2 +z_3^2+z_4^2 > 1, z_1 > 0, z_2 > 0, z_3 > 0, z_4 >0 \rbrace.$$

We apply Algorithm~\ref{algorithm1}. By Remark~\ref{algosimple}, the algorithm runs through $\begin{pmatrix} \alpha & \beta \\ \gamma & \delta \end{pmatrix} \in G \setminus G_{i_4}$ with $\vert \alpha \vert = 1$ and  $\gamma =0$. Hence, the algorithm runs through the matrices of the form $\begin{pmatrix} i & \beta \\ 0 & i \end{pmatrix}$, with $\vert i \vert = 1$ and $\beta$ in $\Gamma_4(\Z)$ such that $i\beta \in \V^4(\Z)$. In the same way as in Remark~\ref{algosimple}, if $M'=UM$, where $U \in G_{i_4}$, then $\Sigma_{M^{-1}}(i_4)=\Sigma_{M'^{-1}}(i_4)$ and hence it is enough to consider matrices of the form $\begin{pmatrix} 1 & \beta \\ 0 & 1 \end{pmatrix}$ with $\beta \in \V^4(\Z)$. The norm of such a matrix is $\vert \beta \vert^2$. It is easy to see that Algorithm~\ref{algorithm1} stops at the norm $K=1$ and returns the set 
$$B= \lbrace z_0+z_1i_1+z_2i_2+z_3i_3+z_4i_4 \mid -\frac{1}{2} < z_i < \frac{1}{2}, i=0,1,2,3 \rbrace.$$

We now turn to Algorithm~\ref{algorithm2}. Observe that $BF=\emptyset$ from the beginning and hence Algorithm~\ref{algorithm2} stops immediately and returns 
$$S= \left\lbrace \begin{pmatrix} i & \beta \\ 0 & i \end{pmatrix} \mid \vert i \vert = 1 , \beta=\pm 1, \pm i_1,\pm i_2,\pm i_3 \right\rbrace.$$ 
By Proposition~\ref{fundgen} and the preceding remarks, the group 
$$\left\langle \begin{pmatrix} 0 & 1 \\ -1 & 0 \end{pmatrix}, \begin{pmatrix} i_1 & 0 \\ 0 & -i_1 \end{pmatrix}, \begin{pmatrix} i_2 & 0 \\ 0 & -i_2 \end{pmatrix} , \begin{pmatrix} i_3 & 0 \\ 0 & -i_3 \end{pmatrix}, \begin{pmatrix} 1 & \beta \\ 0 & 1 \end{pmatrix} \mid \beta=\pm 1, \pm i_1,\pm i_2,\pm i_3  \right\rangle$$
is of finite index in $G$. 

Finally we apply the map $\chi$ to pull back the elements to $\SL_2(\h(\Z))$. We hence find that 
$$\left\langle \begin{pmatrix} 0 & 1 \\ -1 & 0 \end{pmatrix}, \begin{pmatrix} i & 0 \\ 0 & -i \end{pmatrix}, \begin{pmatrix} j & 0 \\ 0 & -j \end{pmatrix} , \begin{pmatrix} k & 0 \\ 0 & -k \end{pmatrix}, \begin{pmatrix} 1 & b \\ 0 & 1 \end{pmatrix} \mid \beta=\pm 1, \pm i,\pm j,\pm k \right\rangle$$
is a subgroup of finite index in $\SL_2(\h(\Z))$. \\

\noindent \textbf{Acknowledgements} I would like to thank Prof.\,Ruth Kellerhals for fruitful discussions on this subject during my stay at Fribourg University in autumn 2014. Moreover I would like to thank Prof.\,Saugata Basu on some helpful insight on algorithms involving semi-algebraic set. Finally I am grateful for the comments of the anonymous referee that helped improving this paper.

\bibliography{cliffordbib}
\bibliographystyle{abbrv}

\end{document}